\documentclass[12pt]{amsart}
\usepackage{amssymb, amsmath}
\usepackage[all]{xy}
\usepackage[inline]{enumitem}
\usepackage{hyphenat}
\usepackage[colorlinks,linkcolor=blue,citecolor=blue,urlcolor=red]{hyperref}
\usepackage{xcolor}
\usepackage{mathtools}

\newcommand{\eq}[2]{\begin{equation}\label{#1}#2 \end{equation}}

\newcommand{\arir}{\ar@{^{(}->}}
\newcommand{\aril}{\ar@{_{(}->}}
\newcommand{\are}{\ar@{>>}}

\newcommand{\xr}[1] {\xrightarrow{#1}}

\newtheorem{lemma}{Lemma}[section]
\newtheorem{thm}[lemma]{Theorem}

\newtheorem{prop}[lemma]{Proposition}

\newtheorem{cor}[lemma]{Corollary}

\theoremstyle{definition}
\newtheorem{defn}[lemma]{Definition}

\newtheorem{para}[lemma]{}

\theoremstyle{remark}

\newtheorem{remark}[lemma]{Remark}

\newtheorem{claim}{Claim}[lemma]
\newtheorem*{claim*}{Claim}

\newcounter{zaehler} 
\numberwithin{equation}{lemma}

\newcommand{\Z}{\mathbb{Z}}

\renewcommand{\P}{\mathbf{P}}
\newcommand{\A}{\mathbf{A}}

\newcommand{\sF}{\mathcal{F}}

\newcommand{\sH}{\mathcal{H}}

\newcommand{\sL}{\mathcal{L}}
\newcommand{\sO}{\mathcal{O}}

\newcommand{\sM}{\mathcal{M}}

\newcommand{\sW}{\mathcal{W}}
\newcommand{\sU}{\mathcal{U}}
\newcommand{\sV}{\mathcal{V}}
\newcommand{\sX}{\mathcal{X}}
\newcommand{\sY}{\mathcal{Y}}
\newcommand{\sZ}{\mathcal{Z}}

\newcommand{\fY}{\mathfrak{Y}}
\newcommand{\fX}{\mathfrak{X}}

\newcommand{\fm}{\mathfrak{m}}

\newcommand{\fp}{\mathfrak{p}}

\newcommand{\Xb}{{\overline{X}}}

\def\sNis{{s\mathrm{Nis}}}
\def\dNis{{d\mathrm{Nis}}}

\newcommand{\Cor}{\operatorname{\mathbf{Cor}}}

\newcommand{\HI}{\operatorname{\mathbf{HI}}}

\newcommand{\RSC}{{\operatorname{\mathbf{RSC}}}}
\newcommand{\RSCNis}{{\operatorname{\mathbf{RSC}}}_{\Nis}}

\newcommand{\ul}[1]{{\underline{#1}}}

\newcommand{\PST}{{\operatorname{\mathbf{PST}}}}

\newcommand{\NST}{\operatorname{\mathbf{NST}}}

\newcommand{\Hom}{\operatorname{Hom}}
\newcommand{\uHom}{\operatorname{\underline{Hom}}}

\newcommand{\Ker}{\operatorname{Ker}}

\newcommand{\Coker}{\operatorname{Coker}}
\newcommand{\Tr}{\operatorname{Tr}}
\newcommand{\Nm}{\operatorname{Nm}}

\newcommand{\Spec}{\operatorname{Spec}}

\newcommand{\Proj}{\operatorname{Proj}}
\newcommand{\Sm}{\operatorname{\mathbf{Sm}}}

\newcommand{\Sch}{\operatorname{\mathbf{Sch}}}
\newcommand{\Ab}{\operatorname{\mathbf{Ab}}}

\newcommand{\tr}{{\operatorname{tr}}}
\newcommand{\Ztr}{{\operatorname{\mathbb{Z}_{\tr}}}}

\newcommand{\eff}{{\operatorname{eff}}}
\newcommand{\fin}{{\operatorname{fin}}}

\newcommand{\red}{{\operatorname{red}}}

\newcommand{\Nis}{{\operatorname{Nis}}}

\newcommand{\et}{{\operatorname{\acute{e}t}}}

\newcommand{\inj}{\hookrightarrow}

\newcommand{\id}{{\operatorname{id}}}

\newcommand{\codim}{{\operatorname{codim}}}
\newcommand{\ch}{{\operatorname{ch}}}

\newcommand{\colim}{\operatornamewithlimits{\varinjlim}}

\newcommand{\ol}{\overline}
\def\lim{\operatornamewithlimits{\varprojlim}}
\renewcommand{\epsilon}{\varepsilon}

\renewcommand{\sp}{{\rm sp}}

\newcommand{\MNST}{\operatorname{\mathbf{MNST}}}

\newcommand{\MCor}{\operatorname{\mathbf{MCor}}}

\newcommand{\MPST}{\operatorname{\mathbf{MPST}}}
\newcommand{\CI}{\operatorname{\mathbf{CI}}}

\newcommand{\CItsp}{\CI^{\tau,sp}}
\newcommand{\CItspNis}{\CI^{\tau,sp}_{\Nis}}

\newcommand{\DMeff}{\operatorname{\mathbf{DM}}^{\eff}}
\newcommand{\CIt}{\CI^{\tau}}

\newcommand{\bcube}{{\ol{\square}}}

\newcommand{\ulMPST}{\operatorname{\mathbf{\underline{M}PST}}}
\newcommand{\ulMNST}{\operatorname{\mathbf{\underline{M}NST}}}

\newcommand{\ulMCor}{\operatorname{\mathbf{\underline{M}Cor}}}
\newcommand{\ulMCorls}{\ulMCor_{ls}}
\newcommand{\ulMCorfin}{\ulMCor^{\fin}}
\newcommand{\ulMCorfinls}{\ulMCor^{\fin}_{ls}}

\newcommand{\uMPST}{\operatorname{\mathbf{\underline{M}PST}}}
\newcommand{\uMNST}{\operatorname{\mathbf{\underline{M}NST}}}
\newcommand{\uMPSTlog}{\operatorname{\mathbf{\underline{M}PST}}_{\log}}
\newcommand{\uMNSTlog}{\operatorname{\mathbf{\underline{M}NST}}_{\log}}

\newcommand{\uMCor}{\operatorname{\mathbf{\underline{M}Cor}}}
\newcommand{\uMCorls}{\ulMCor_{ls}}
\newcommand{\uMCorfin}{\ulMCor^{\fin}}
\newcommand{\uMCorfinls}{\ulMCor^{\fin}_{ls}}

\newcommand{\ulomega}{\underline{\omega}}

\newcommand{\ulomegaCI}{\underline{\omega}^{\CI}}

\newcommand{\lDMeff}{\operatorname{\mathbf{logDM}}^{\eff}}

\def\mc{\mathrm{mc}}
\def\cd#1#2{{#1}_{(#2)}}
\def\isom{\overset{\simeq}\longrightarrow}
\def\tF{\tilde{F}}

\def\tL{\tilde{L}}
\def\tE{\tilde{E}}
\def\hL{\hat{L}}

\def\hY{\hat{Y}}
\def\hZ{\hat{Z}}
\def\tZ{\tilde{Z}}
\def\tD{\tilde{D}}

\def\sYred{\sY_{\red}}
\def\rmapo#1{\overset{#1}{\longrightarrow}}
\def\qfor{\;\text{ for }}
\def\qwith{\;\text{ with }}
\def\qaq{\;\text{ and }\;}
\def\cd#1#2{{#1}^{(#2)}}
\def\hen#1#2{#1^{h}_{|#2}}

\def\Fc{F_{-1}}
\def\Fcgm{F_{-1}^{(1)}}
\def\Gc{G_{-1}}
\def\Gcgm{G_{-1}^{(1)}}

\def\Fcontt#1{F^{(#1)}_{-1}}

\def\hcubespNis{h_{0,\Nis}^{\bcube,\sp}}

\def\Hlog#1{H_{\log}^{#1}}
\def\tauu#1{\tau^{(#1)}}
\def\triv{\mathrm{triv}}

\def\aVNis{a^V_\Nis}
\def\ulaNis{\ul{a}_\Nis}

\newcommand{\lPST}{{\operatorname{\mathbf{PSh}}^{ltr}}}
\newcommand{\dNST}{{\operatorname{\mathbf{Shv}}^{ltr}_{dNis}}}
\newcommand{\lSm}{{\operatorname{l\mathbf{Sm}}}}
\newcommand{\lCor}{{\operatorname{l\mathbf{Cor}}}}
\newcommand{\SmlSm}{{\operatorname{\mathbf{Sm}l\mathbf{Sm}}}}
\newcommand{\lCorSm}{\lCor_{\SmlSm}}
\def\fs{\mathrm{fs}}
\def\Log{\mathcal{L}og}
\def\pX{\partial \fX}
\def\pY{\partial \fY}

\def\fXMP{\fX^{MP}}
\def\Flog{F^{\log}}
\def\Glog{G^{\log}}
\def\alog{\alpha^{*\log}}

\def\tSm{\Sm^{\rm pro}}

\def\cd#1#2{{#1}^{(#2)}}

\def\kX#1{K^h_{X,#1}}
\def\kY#1{K^h_{Y,#1}}

\def\kh#1#2{K^h_{#1,#2}}
\def\Oh#1#2{\sO^h_{#1,#2}}
\def\Lfinls{\Lambda^{\fin}_{ls}}

\title{Reciprocity sheaves and logarithmic motives}
\author{Shuji Saito}

\address{Graduate School of Mathematical Sciences, University of Tokyo, 3-8-1 Komaba, Tokyo 153-8941, Japan}
\email{sshuji@msb.biglobe.ne.jp}
\thanks{The author is supported by the JSPS KAKENHI Grant (20H01791). }

\begin{document}
\maketitle
\tableofcontents

\begin{abstract}
We connect two developments aiming at extending Voevodsky's theory of motives over a field in such a way to encompass non-$\A^1$-invariant phenomina.
One is theory of \emph{reciprocity sheaves} introduced by Kahn-Saito-Yamazaki.
Another is theory of the triangulated category $\lDMeff$ of \emph{logarithmic motives} launched by Binda, Park and \O stv\ae r.
We prove that the Nisnevich cohomology of reciprocity sheaves is representable in
$\lDMeff$.
\end{abstract}

\section*{Introduction}

We fix once and for all a perfect base field $k$. 
The main purpose of this paper is to connect two developments aiming at extending Voevodsky's theory of motives over $k$ in such a way to encompass non-$\A^1$-invariant phenomina.
One is theory of \emph{reciprocity sheaves} introduced by Kahn-Saito-Yamazaki (\cite{ksyI} and \cite{ksyII}) and developed in \cite{shuji} and \cite{brs}. 
Voevodsky's theory is based on the category $\PST$ of \emph{presheaves with transers}, defined as the category of additive presheaves of abelian groups on 
the category $\Cor$ of finite correspondences:
$\Cor$ has the same objects as the category $\Sm$ of separated smooth schemes of finite type over $k$ and morphisms in $\Cor$ are finite correspondences. 
Let $\NST\subset \PST$ be the full subcategory of Nisnevich sheaves, i.e.
those objects $F\in \PST$ whose restrictions $F_X$ to the small \'etale site $X_{\et}$ over $X$ are Nisnevich sheaves for all $X\in \Sm$. 
Voevodsky proved that $\NST$ is a Grothendieck abelian category and defined the triangulated category $\DMeff$ of effective motives as the localization of 
the derived category $D(\NST)$ of complexes in $\NST$ with respect to an \emph{$\A^1$-weak equivalence} (see \cite[Def. 14.1]{mvw}).
It is equipped with a functor $M: \Sm \to \DMeff$ associating the motive $M(X)$ of $X\in \Sm$. 

Let $\HI_\Nis\subset\NST$ be the full subcategory consisting of $\A^1$-invariant objects, namely such $F\in \NST$ that the projection $\pi_X: X\times\A^1\to X$ induces an isomorphism $\pi_X^* :F(X)\simeq F(X\times\A^1)$ for any $X\in \Sm$. We say that $F\in \HI_\Nis$ is strictly $\A^1$-invariant if $\pi_X$ induces
isomorphisms
\[ \pi_X^* : H^i_\Nis(X,F_X) \simeq H^i_\Nis(X\times\A^1,F_{X\times\A^1})
\;\text{ for all } i\geq 0.\]
The following theorem plays a fundamental role in Voevodsky's theory.

 \begin{thm}\label{Voevodsky}(Voevodsky \cite{voetri})
Any $F\in \HI_\Nis$ is strictly $\A^1$-invariant and we have a natural isomorphism
\eq{voe}{H^i_\Nis(X,F_X)\simeq \Hom_{\DMeff}(M(X),L^{\A^1}F[i])\qfor X\in \Sm,}
where $L^{\A^1} : D(\NST)\to \DMeff$ is the localization functor.
\end{thm}
\medbreak

Notice that there are interesting and important objects of $\NST$ which do not belong to $\HI_\Nis$. Such examples are given by the sheaves $\Omega^i$ of (absolute or relative) differential forms, and the $p$-typical de Rham-Witt sheaves $W_m\Omega^i$ of Bloch-Deligne-Illusie, and smooth commutative $k$-groups schemes with a unipotent part (seen as objects of $\NST$),  and the complexes $R\epsilon_* \Z/p^r(n)$ in case $\ch(k)=p>0$, where $\Z/p^r(n)$ is the \'etale motivic complex of weight $n$ with $\Z/p^r$ coefficients and
$\epsilon$ is the change of site functor from the \'etale to the Nisnevich topology. For such examples, \eqref{voe} fails to hold since $\pi_X:X\times\A^1\to X$ induces an isomorphism $M(X\times\A^1)\simeq M(X)$ in $\DMeff$ but the maps induced on cohomology of those sheaves are not isomorphism.

The category $\RSC_{\Nis}$ of reciprocity sheaves is a full abelian subcategory of $\NST$ that contains $\HI_{\Nis}$ as well as the non-$\A^1$-invariant objects mentioned above. Heuristically, its objects satisfy the property that for any $X\in \Sm$, each section $a\in F(X)$ ``has bounded ramification at infinity'' and the objects of $\HI_\Nis$ are special reciprocity sheaves with the property that every section $a\in F(X)$ has ``tame'' ramification at infinity\footnote{This heuristic viewpoint is manifested in \cite[Th. 2]{rs}.}.
Slightly more exotic examples of reciprocity sheaves are given by the sheaves ${\rm Conn}^1$ (in case $\ch(k)=0$), whose sections over $X$ are rank $1$-connections, or ${\rm Lisse}^1_\ell$ (in case $\ch(k)=p>0$), whose sections on $X$ are the lisse $\overline{\mathbb{Q}}_\ell$-sheaves of rank $1$. Since $\RSC_{\Nis}$ is an abelian category equipped with a lax symmetric monoidal structure by \cite{rsy}, many more interesting examples can be manufactured by taking kernels, quotients and tensor products (see \cite[\S11.1]{brs} for more examples).
\medbreak

The main purpose of this article is to establish the formula \eqref{voe} for all $F\in \RSC_\Nis$ in a new category which enlarges $\DMeff$ (see \eqref{CItspNis-logDMintro}). It is the triangulated category $\lDMeff$ of \emph{logarithmic motives} introduced  by Binda, Park and \O stv\ae r in \cite{bpo}. 
Let $\lSm$ be the category of log smooth and separated $\fs$ log schemes of finite type over $k$ and $\lCor$ be the category with the same objects as $\lSm$ and whose morphisms are log finite correspondences (see \cite[Def. 2.1.1]{bpo}).
Let $\lPST$ be the category of additive presheaves of abelian groups on $\lCor$ and $\dNST\subset \lPST$ be the full subcategory consisting of those 
$\sF$ whose restrictions to $\lSm$ are dividing Nisnevich sheaves (see \cite[Def. 3.1.4]{bpo}). It is shown in \cite[\S 4 and Pr. 4.6.6]{bpo} that $\dNST$ is a Grothendieck abelian category, and $\lDMeff$ is defined as 
the localization of the derived category $D(\dNST)$ of complexes in $\dNST$ with respect to a \emph{$\bcube$-weak equivalence}, 
where $\bcube$ is $\P^1$ with the log-structure associated to the effective divisor $\infty\hookrightarrow \P^1$ (see \cite[Def. 5.2.1]{bpo}\footnote{
In fact it is defined in loc.cite. as the localization of the homotopy category of complexes in $\dNST$ with respect to a $\bcube$-local descent model structure.}). 
It is equipped with a functor $M: \lSm \to \lDMeff$ associating the logarithmic motive $M(\fX)$ of $\fX\in \lSm$. 

Now we can state the main result of this paper.

\begin{thm}\label{CItspNis-logDMintro}(Theorems \ref{CItspNis-logDM} and \ref{Log-ffexact})
There exists an exact and fully faithful functor
\eq{LogRSCintro}{ \Log : \RSC_\Nis \to \dNST\;:\; F\to \Flog=\Log(F) }
such that $\Flog$ for $F\in \RSCNis$ is strictly $\bcube$-invariant in the sense \cite[Def. 5.2.2]{bpo}\footnote{It is an logarithmic analogue of Voevodsky's strict $\A^1$-invariance.}. For $X\in \Sm$ we have a natural isomorphism
\eq{CItspNis-logDMintro}{ H^i_\Nis(X,F_X)\simeq \Hom_{\lDMeff}(M(X,\triv),L^{\bcube}\Flog[i]),}
where $L^{\bcube} : D(\dNST)\to \lDMeff$ is the localization functor and
$(X, \triv)$ is the log-scheme with the trivial log-structure.

\end{thm}

We remark (see Remark \ref{rem:logcohRSC}) that for $F=\Omega^i$, $\Flog(\fX)$ for $\fX\in \lSm$ whose underlying scheme is smooth, agrees 
with the sheaf of logarithmic differential forms of $\fX$ at least assuming $\ch(k)=0$
\footnote{The assumption is necessary to use \cite[Cor. 6.8]{rs} proved in case $\ch(k)=0$. We expect that it is removed by using a forthcoming work of K. R\"ulling extending \cite[Cor. 6.8]{rs} to the case $\ch(k)>0$.}.
\medbreak

We now explain the organization of the paper.

In \S \ref{preliminaries} we discuss some preliminaries and fix the notation. 
We recall the definitions and basic properties of \emph{modulus (pre)sheaves with transfers} from \cite{kmsyI}, \cite{kmsyII}, \cite{ksyII} and \cite{shuji}. 
It is a generalization of Voevodsky's (pre)sheaves with transfers to a version with modulus. 
The category $\uMCor$ of \emph{modulus correspondences} is introduced.
Its objects are pairs $\sX = (\ol{X}, D)$, where $\ol{X}$ is a separated scheme of finite type over $k$ equipped with an effective Cartier divisor $D$ such that the \emph{interior} $\ol{X}-D = X$ is smooth. The morphisms are finite correspondences on the interiors satisfying some admissibility and a properness condition. 
Let $\uMPST$ be the category of additive presheaves of abelian groups on $\uMCor$. A full subcategory $\uMNST\subset \uMPST$ of \emph{Nisnevich sheaves} is defined and there is a functor (see \S \ref{preliminaries}\eqref{omegaCI})
\[\ulomegaCI : \RSC_\Nis \to \uMNST.\]
For every $F\in \RSC_\Nis$ and $X\in \Sm$, it provides an exhaustive filtration on the group $F(X)$ of sections over $X$ which measures depth of ramification along a boundary of a partial compactification of $X$:
For $(\Xb,D)\in \uMCor$ with $\Xb-D=X$, we get
the subgroups $\tF(\Xb,D)\subset F(X)$ with $\tF=\ulomegaCI F$ such that $\tF(\Xb,D_1)\subset \tF(\Xb,D_2)$ if $D_1\leq D_2$.

In \S \ref{puritymodulus} we prove as a key technical input an analogue of Zariski-Nagata's purity theorem (\cite[X 3.4]{SGA2}) for $\tF(\Xb,D)$ as above. It asserts the exactness of the sequence
\[ 0\to \tF(\Xb,D) \to F(X) \to 
\underset{\xi\in \cd D 0}{\bigoplus} \frac{F(\hen {\Xb} \xi-\xi)}{\tF(\hen {\Xb} \xi,\xi)},\]
in case $\Xb\in \Sm$ and $D$ is reduced simple normal crossing divisor,
where $\cd D 0$ is the set of the irreducible components of $D$ and 
$\hen {\Xb} \xi$ is the henselization of $\Xb$ at $\xi$.
In \cite{rs-RSCpairing}, this result is generalized to the case where $D$ may not be reduced under the assumption that $\Xb$ admits a smooth compactification.

In \S \ref{HLS} we review \emph{higher local symbols} for reciprocity sheaves constructed in \cite{rs-RSCpairing}.
It is an effective tool with which one can decide when a given element of $F(X)$ with $F\in \RSC_\Nis$ and $X\in \Sm$ belongs to $\tF(\Xb,D)$ as above. 
The construction of the pairing depends on pushforward maps for cohomology of reciprocity sheaves constructed in \cite{brs} (which means that Theorem \ref{CItspNis-logDMintro} depends on the result of \cite{brs}).

In \S \ref{logcoh}, we prove the following result:
Let $\uMCorfinls$ be the subcategory of $\uMCor$ whose objects are pairs $(X,D)$ such that $X\in \Sm$ and the reduced parts $D_{\red}$ of $D$ is a SNCD on $X$ and whose morphisms are modulus correspondences satisfying a finiteness conditions
instead of the properness condition (see \S \ref{preliminaries}\eqref{MCorfin}). Then, for $F\in \RSC_\Nis$, the association
\[ \tF^{\log}: (X,D) \to \ulomegaCI F (X,D_{\red})\]
gives a presheaf on $\uMCorfinls$, which gives rise to a cohomology theory
$H^i_{\log}(-,\tF^{\log})$ on $\uMCorfinls$, called  
\emph{the $i$-th logarithmic cohomology with coefficient $F$} (see Definition \ref{def:logcoh}). The higher local symbols for $F$ plays a fundamental role in the proof of the result .

In \S \ref{invariance-bu}, we prove the invariance of logarithmic cohomology under blowups: Let $\Lfinls$ be the subcategory of $\uMCorfinls$ whose objects are the same as $\uMCorfinls$ and whose morphisms are those $\rho: (Y,E) \to (X,D)$ where $E=\rho^* D$ and $\rho$ are induced by blowups of $X$ in smooth centers $Z\subset D$ which are normal crossing to $D$ (see the beginning of the section).
Then, for $F\in \RSC_\Nis$ and $\rho:\sY\to \sX$ in $\Lfinls$, we have
\[ \rho^*: \Hlog i(\sX,F) \cong \Hlog i(\sY,F) \qfor \forall i\geq 0.\]

In \S \ref{log-mottives}, we prove Theorem \ref{CItspNis-logDMintro}, which is 
a formal consequence of the theorems in \S \ref{logcoh} and \S \ref{invariance-bu}.

\bigskip\noindent
{\bf Acknowledgements.} The author would like to thank Kay R\"ulling and F. Binda for many valuable discussions and comments.
He is also grateful to A. Merici to whom he owes crucial ideas for \S\ref{puritymodulus}.

\section{Preliminaries}\label{preliminaries}

We fix once and for all a perfect base field $k$.
In this section we recall the definitions and basic properties of modulus sheaves with transfers from \cite{kmsyI} and \cite{shuji} (see also \cite{ksyII} for a more detailed summary).
\begin{enumerate}
    \item Denote by $\Sch$ the category of separated schemes of finite type over $k$ and by $\Sm$ the full subcategory of smooth schemes. For $X, Y \in \Sm$, an integral closed subscheme of $X\times Y$ that is finite and surjective over a connected component of $X$ is called a \emph{prime correspondence from $X$ to $Y$}. 
The category $\Cor$ of finite correspondences has the same objects as $\Sm$, and for $X, Y \in \Sm$, $\Cor(X,Y)$ is the free abelian group on the set of all prime correspondences from $X$ to $Y$ (see \cite{voetri}). We consider $\Sm$ as a subcategory of $\Cor$ by regarding a morphism in $\Sm$ as its graph in $\Cor$. 

Let $\PST$ be the category of additive presheaves of abelian groups on $\Cor$ whose objects are called \emph{presheaves with transfers}. 
Let $\NST\subseteq \PST$ be the category of Nisnevich sheaves with transfers and let 
\[\aVNis :\PST\to \NST\]
be Voevodsky's Nisnevich sheafification functor, which is an exact left adjoint to the inclusion $\NST\to \PST$.
Let $\HI\subseteq \PST$ be the category of $\A^1$-invariant presheaves and put  $\HI_\Nis=\HI\cap \NST\subseteq \NST$. 

\item\label{Corpro}
Let $\tSm$ be the category of $k$-schemes $X$ which are essentially smooth over $k$, i.e. $X$ is a limit $\lim_{i \in I} X_i$ over a filtered set $I$, 
where $X_i $ is smooth over $k$ and all transition maps are \'etale.
Note $\Spec K\in \tSm$ for a function field $K$ over $k$ thanks to the assumption that $k$ is perfect.
We define $\Cor^{\rm pro}$ whose objects are the same as $\tSm$ and morphisms are defined as \cite[Def. 2,2]{rs}.
We extend $F\in \PST$ to a presheaf on $\Cor^{\rm pro}$ by
$F(X) := \colim_{i \in I} F(X_i)$ for $X$ as above.

\item\label{uMCor}
 We recall the definition of the category $\ulMCor$ from \cite[Definition 1.3.1]{kmsyI}. A pair $\sX = (X,D)$ of $X \in \Sch$ and an effective Cartier divisor $D$ on $X$ is called a \emph{modulus pair} if $M - |M_\infty| \in \Sm$.  
Let $\sX=(X,D_X)$, $\sY=(Y,D_Y)$ be modulus pairs and $\Gamma \in \Cor(X-D_X,Y-D_Y)$ be a prime correspondence. Let $\overline{\Gamma} \subseteq X \times Y$ be the closure of $\Gamma$, and let $\overline{\Gamma}^N\to X\times Y$ be the normalization. We say $\Gamma$ is \emph{admissible} (resp. \emph{left proper}) if $(D_X)_{\overline{\Gamma}^N}\geq (D_Y)_{\overline{\Gamma}^N}$
(resp. if $\overline{\Gamma}$ is proper over $X$). Let $\ulMCor(\sX, \sY)$ be the
subgroup of $\Cor(X-D_X,Y-D_Y)$ generated by all admissible left proper prime correspondences. The category $\ulMCor$ has modulus pairs as objects and $\ulMCor(\sX, \sY)$ as the group of morphisms from $\sX$ to $\sY$.

\item
Let $\ulMCorls\subset \ulMCor$ be the full subcategory of $(X,D)\in \ulMCor$ with $X\in \Sm$ and $|D|$ a normal crossing divisor on $X$. 

\item\label{MCorfin}
Let $\ulMCorfin\subset \ulMCor$ be the full subcategory of the same objects such that $\ulMCorfin(\sX, \sY)$ are generated by all admissible \emph{finite} prime correspondences, where finite prime correspondences are defined by replacing the left properness in \eqref{uMCor} by finiteness.
We also define $\ulMCorfinls\subset \ulMCorfinls\cap \ulMCorls$.

\item
There is a canonical pair of adjoint functors $\lambda\dashv \ulomega$:\[
\lambda: \Cor\to \ulMCor\quad X\mapsto (X,\emptyset),\]
\[
\ulomega: \ulMCor\to\Cor\quad (X,D)\mapsto X-|D|,
\]

\item
There is a full subcategory $\mathbf{MCor}\subset \ulMCor$ consisting of 
\emph{proper modulus pairs}, where
a modulus pair $(X,D)$ is \emph{proper} if $X$ is proper.
Let $\tau:\MCor\hookrightarrow \ulMCor$ be the inclusion functor and $\omega=\ulomega\tau$.


\item 
Let $\MPST$ (resp. $\uMPST$) be the category of additive presheaves of abelian groups on $\MCor$ (resp. $\uMCor$) whose objects are called \emph{modulus presheaves with transfers}. 
For $\sX \in \MCor$, let $\Ztr(\sX)=\ulMCor(-,\sX)$ be the representable object of $\ulMPST$. We sometime write $\sX$ for $\Ztr(\sX)$ for simplicity.

\item\label{ulMCorpro}
By the same manner as \eqref{Corpro}, the category $\ulMCor^{\rm pro}$ is defined
and $F\in \ulMPST$ is extended to a presheaf on $\ulMCor^{\rm pro}$
(see \cite[\S 3.7]{rs}).

\item\label{ulomega} The adjunction $\lambda\dashv \ulomega$ induce a string of $4$ adjoint functors $(\lambda_!=\ulomega^!,\lambda^*=\ulomega_!,\lambda_*=\ulomega^*,\ulomega_*)$ (see \cite[Pr. 2.3.1]{kmsyI}):
\[\ulMPST\begin{smallmatrix}\ulomega^!\\\longleftarrow\\\ulomega_!\\\longrightarrow\\\ulomega^*\\\longleftarrow\\\ulomega_*\\ \longrightarrow\end{smallmatrix}\PST \]
where $\ulomega_!,\ulomega_*$ are localisations and $\ulomega^!$ and $\ulomega^*$ are fully faithful.


\item\label{tau}
The functor $\tau$ yields a string of $3$ adjoint functors $(\tau_!,\tau^*,\tau_*)$:
\[\MPST\begin{smallmatrix}\tau_!\\\longrightarrow\\\tau^*\\\longleftarrow\\\tau_*\\ \longrightarrow\end{smallmatrix}\ulMPST \]
where $\tau_!,\tau_*$ are fully faithful and $\tau^*$ is a localisation; $\tau_!$ 
has a pro-left adjoint $\tau^!$, hence is exact (see \cite[Pr. 2.4.1]{kmsyI}). We will denote by $\ulMPST^{\tau}$ the essential image of $\tau_!$ in $\ulMPST$. 

\item The modulus pair $\bcube:=(\P^1,\infty)$ has an interval structure induced by the one of $\A^1$ (see \cite[Lem. 2.1.3]{ksyII}). We say $F \in \MPST$ is $\bcube$-invariant if $p^* : F(\sX) \to F(\sX \otimes \bcube)$ is an isomorphism for any $\sX \in \MCor$, where $p : \sX \otimes \bcube \to \sX$ is the projection.
Let $\CI$ be the full subcategory of $\MPST$ consisting of all $\bcube$-invariant objects and $\CIt\subset \ulMPST$ be the essential image of $\CI$ under $\tau_!$.

\item\label{CI}
Recall from \cite[Theorem 2.1.8]{ksyII} that $\CI$ is a Serre subcategory of $\MPST$,
and that the inclusion functor $i^\bcube : \CI \to \MPST$
has a left adjoint $h_0^\bcube$ and a right adjoint $h^0_\bcube$ given 
for $F \in \MPST$ and $\sX \in \MCor$ by
\begin{align*} 
&h_0^\bcube(F)(\sX)
=\Coker(i_0^* - i_1^* : F(\sX \otimes \bcube) \to F(\sX)),
\\
\notag
&h^0_\bcube(F)(\sX)=\Hom(h_0^\bcube(\sX), F).
\end{align*}
For $\sX\in \MCor$, we write $h_0^\bcube(\sX)=h_0^\bcube(\Ztr(\sX))\in \CI$, and 
by abuse of notation, we let $h_0^\bcube(\sX)$ denote also for $\tau_!h_0^\bcube(\sX)\in \CIt$.

\item For $F\in \ulMPST$ and $\sX=(X,D)\in \ulMCor$, write $F_{\sX}$ for the presheaf
on the small \'etale site $X_{\et}$ over $X$ given by $U\to F(\sX_U)$ for $U\to X$ \'etale, where $\sX_U=(U,D_{|U})\in \ulMCor$. We say $F$ is a Nisnevich sheaf if so is $F_{\sX}$ for all $\sX\in \ulMCor$ (see \cite[Section 3]{kmsyI}).
We write $\ulMNST\subset \ulMPST$
for the full subcategory of Nisnevich sheaves and put
\[\MNST^\tau=\ulMNST\cap \MPST^\tau,\quad \CIt_\Nis =\CIt\cap \MNST^\tau.\]
By \cite[Prop. 3.5.3]{kmsyI} and \cite[Theorem 2]{kmsyII}, the inclusion functor
$i_\Nis: \ulMNST\to \ulMPST$ has an exact left adjoint $\ulaNis$ such that
$\ulaNis(\MPST^\tau)\subset \MNST^\tau$.
The functor $\ulaNis$ has the following description: For $F\in \ulMPST$ and $\sY\in \ulMCor$, let $F_{\sY,\Nis}$ be the usual Nisnevich sheafification of $F_{\sY}$. 
Then, for $(X,D)\in \ulMCor$ we have\[
\ulaNis F (X,D) = \colim_{f:Y\to X} F_{(Y,f^*D),\Nis}(Y)
\]
where the colimit is taken over all proper maps $f:Y\to X$ that induce isomorphisms $Y-|f^*D|\xrightarrow{\sim}X-|D|$.

\item\label{NST} By \cite[Pr. 6.2.1]{kmsyII}, $\ulomega^*$ and $\ulomega_!$ from \eqref{ulomega} respect $\ulMNST$ and $\NST$ and induce a pair of adjoint functors (which for simplicity we write $\ulomega_!$ and $\ulomega^*$). Moreover, we have
\[\ulomega_!\ulaNis=\aVNis\ulomega_!.\]
By \cite[Lem. 2.3.1]{ksyII} and \cite[Pr. 6.2.1a)]{kmsyII}, for $F\in \PST$, we have $F\in \HI$ (resp $F\in \HI_\Nis$) if and only if $\ulomega^*F\in \CIt$ (resp $\ulomega^*F\in \CIt_\Nis$).

\item\label{semipure} We say that $F\in \ulMPST$ is \emph{semi-pure} if the unit map
\[ u: F\to \ulomega^*\ulomega_!F \]
is injective. For $F\in \ulMPST$ (resp. $F\in \ulMNST$),  let $F^{sp}\in \ulMPST$ (resp. $F^{sp}\in \ulMNST$) be the image of
$F\to \ulomega^*\ulomega_! F$ (called the semi-purification of $F$. See \cite[Lem. 1.30]{shuji}).
For $F\in \ulMPST$  we have
\[ \ulaNis(F^{sp}) \simeq (\ulaNis F)^{sp}.\]
This follows from the fact that $\ulaNis$ is exact and commutes with 
$\ulomega^*\ulomega_!$.
For $F\in \MPST^{\tau}$ we have $F^{sp}\in \MPST^{\tau}$ since $\tau$ is exact and $\ulomega^*\ulomega_!\tau_!=\tau_! \omega^*\omega_!$.

\item\label{CItspNis} 
Let $\CItsp\subset\CIt$ be the full subcategory of semipure objects and
consider the full subcategory
\[\CItsp_\Nis =\CItsp\cap \MNST^\tau \subset \CIt_\Nis.\]
By \cite[Th. 0.1 and 0.4]{shuji}, we have $\ulaNis(\CItsp)\subset \CItsp_\Nis$.



\item\label{RSC} 
We write  $\RSC\subseteq \PST$ for the essential image of $\CI$ under
$\omega_!$ (which is the same as the essential image of $\CItsp$ under
$\ulomega_!$ since $\omega_!=\ulomega_!\tau_!$ and $\ulomega_!F=\ulomega_!F^{sp}$).
Put $\RSC_{\Nis}=\RSC\cap \NST$.
The objects of $\RSC$ (resp. $\RSC_\Nis$) are called reciprocity presheaves 
(resp. sheaves). By \cite[Th. 0.1]{shuji}, we have
\begin{equation}\label{aVNisRSC}
\aVNis(\RSC)\subset \RSC_\Nis.
\end{equation}
We have $\HI\subseteq \RSC$ and it contains also smooth commutative group schemes (which may have non-trivial unipotent part), and the sheaf $\Omega^i$ of K\"ahler differentials, and the de Rham-Witt sheaves $W\Omega^i$ (see \cite{ksyI} and \cite{ksyII}).

\item
$\NST$ is a Grothendieck abelian category by \cite[Lem. 3.1.6]{voetri} and 
we can make $\RSCNis$ its full sub-abelian category as follows:
We define the kernel (resp. cokernel) of a map $\phi: F\to G$ in $\RSCNis$ to be that of $\phi$ as a map in $\NST$. Here we need \eqref{aVNisRSC} to ensure
that the cokernel of $\phi$ in $\NST$ stays in $\RSCNis$.
By definition, a sequence $0\to F\to G\to H\to 0$ is exact in $\RSCNis$ if and only if it is exact in $\NST$.

\item\label{omegaCI}
\def\hF{\hat{F}}
By \cite[Prop. 2.3.7]{ksyII} we have a pair of adjoint functors:
\begin{equation}\label{omegaCIadjoint0}
\CI\begin{smallmatrix}\omega^{\CI}\\\longleftarrow\\\omega_!\\\longrightarrow
\end{smallmatrix}\RSC, 
\end{equation}
where $\omega^{\CI}=h^0_\bcube\omega^*$ and it is fully faithful.
It induces a pair of adjoint functors:
\begin{equation}\label{omegaCIadjoint}
\CIt\begin{smallmatrix}\ulomega^{\CI}\\\longleftarrow\\\ulomega_!\\\longrightarrow
\end{smallmatrix}\RSC, 
\end{equation}
where $\ulomega^{\CI}=\tau_!h^0_\bcube\omega^*$ and it is fully faithful.
Indeed, let $F=\tau_!\hF$ for $\hF\in \CI$ and $G \in \RSC$. 
In view of \eqref{CI} and the exactness and full faithfulness of $\tau_!$, we have
\begin{multline*}
    \Hom_{\CIt}(F,\tau_!h^0_\bcube\omega^* G) \simeq  
\Hom_{\CI}(\hF,h^0_\bcube\omega^*G) \simeq \\
    \Hom_{\MPST}(\hF,\omega^*G) \simeq
\Hom_{\ulMPST}(\tau_! \hF,\ulomega^*G) \simeq \Hom_{\RSC}(\ulomega_! F,G).
\end{multline*}
In view of \eqref{NST}, \eqref{omegaCIadjoint} induce pair of adjoint functors:
\begin{equation}\label{omegaCIadjointNis}
\CItsp_\Nis\begin{smallmatrix}\ulomega^{\CI}\\\longleftarrow\\\ulomega_!\\\longrightarrow \end{smallmatrix}\RSC_\Nis, 
\end{equation}

\end{enumerate}
\bigskip

\section{Purity with reduced modulus}\label{puritymodulus}

For $F\in \ulMPST$, we put
\[ \Fc=\Ker\big(\uHom_{\ulMPST}((\P^1-0,\infty),F)\rmapo{i_1^*} F\big),\]
\[ \Fcgm=\Ker\big(\uHom_{\ulMPST}((\P^1,0+\infty),F)\rmapo{i_1^*} F\big),\]
Note that if $F\in \CItspNis$, we have for $\sX\in \uMCor$
\eq{eq;Fgm}{
\begin{aligned}
&\Fcgm(\sX) = \Hom_{\ulMPST}(\hcubespNis(\P^1,0+\infty)^0,\uHom_{\ulMPST}(\Ztr(\sX),F)),\\
&\Fc(\sX) =\colim_n \Hom_{\ulMPST}(\hcubespNis(\P^1,n\cdot 0+\infty)^0,\uHom_{\ulMPST}(\Ztr(\sX),F)),\end{aligned}}
where 
\[\hcubespNis(\P^1,n\cdot 0+\infty)^0=\Coker\big(\Z=\Ztr(\Spec k,\emptyset) 
\rmapo{i_1} \hcubespNis(\P^1,n\cdot 0+\infty)\big).\]

\begin{defn}\label{def;tauu}
For $e_1,\dots,e_r\in \{0,1\}$, put
\[ \tauu {e_1,\dots,e_r}F =\tauu {e_r}\cdots \tauu {e_1}F,\]
where
\[\tauu 0 F=\Fc\qaq  \tauu 1 F = \Fc/\Fcgm.\]

\end{defn}

The existence of retractions in the following lemma was suggested by A. Merici.
It implies $\tauu {e_1,\dots,e_r}F\in \CItspNis$ if $F\in\CItspNis$.

\begin{lemma}\label{lem:F-1splitting}
For $F\in \CItspNis$, the inclusion $\Fcgm \to \Fc$ admits a retraction
$s_F : \Fc\to \Fcgm$ such that for any map $\phi: F\to G$ in $\CItspNis$, the following diagram is commutative:
\[\xymatrix{\Fc\ar[r]^{s_F}\ar[d]^{\phi}  & \Fcgm \ar[d]^{\phi} \\
\Gc\ar[r]^{s_F}  & \Gcgm  \\}\]
In particular $\tauu 1 F\in \CItspNis$ if $F\in \CItspNis$.
\end{lemma}
\begin{proof}
In view of \eqref{eq;Fgm}, this follows from \cite[Lem. 2.4]{brs}.
\end{proof}

\begin{thm}\label{thm;purityM}
Let $F\in \CItspNis$.
Let $\sX=\Spec K\{t_1,\dots,t_n\}$ and $D= \{t_1^{e_1}\cdots t_n^{e_n}=0\}\subset \sX$ with $e_1,\dots,e_n\in \{0,1\}$. 
For a subset $I\subset [1,n]$ let $i_\sH : \sH \hookrightarrow\sX$ be the closed immersion defined by $\{t_i=0\}_{i\in I}$ and 
$D_\sH=\{\underset{ j\in [1,n]-I}{\prod} t_j^{e_j}=0\}\subset \sH$. 
Then 
\begin{equation}\label{eq;lem;purityM}
 R^\nu i_\sH^! F_{(\sX,D)} =0 \qfor \nu\not=q:=|I|,
\end{equation}
and there is an isomorphism
\begin{equation}\label{eq1;lem;purityM}
 (\tauu {e_I} F)_{(\sH,D_\sH)} \simeq R^q i_\sH^! F_{(\sX,D)}\qwith 
e_I=(e_i)_{i\in I}\in \Z_{\geq 0}^q.
\end{equation}
\end{thm}
\begin{proof}

The proof is divided into two steps.

\medbreak\noindent{\bf Step 1:}\;
We prove \eqref{eq;lem;purityM} and \eqref{eq1;lem;purityM} in case $q=|I|=1$. 

For $\nu=0$ \eqref{eq;lem;purityM} follows from the semipurity of $F$.
Thus it suffices to show \eqref{eq;lem;purityM} only for $\nu>1$.
Let $J=\{j\in [1,n]\;|\; e_j\not=0\}$ and $r=|J|$.
If $\dim(\sX)=0$, the assertion is trivial. If $r=0$, the assertion follows from \cite[Cor. 8.6(3)]{shuji}.
Assume $r>0$ and $\dim(\sX)\geq 1$, and proceed by the double induction on $r$ and $\dim(\sX)$. Without loss of generality, we may assume 
\begin{enumerate}
\item[$(\spadesuit)$]
$e_1\not=0$, and $\sH=\{t_1=0\}$ if $\sH\subset |D|$.
\end{enumerate}
Let $\iota:\sZ\hookrightarrow \sX$ be the closed immersion defined by $\{t_1=0\}$
and $D_\sZ=\{t_2^{e_2}\cdots t_r^{e_r}=0\}\subset \sZ$ and $D'=\{t_2^{e_2}\cdots t_r^{e_r}=0\}\subset \sX$.
By \cite[Lem. 7.1]{shuji}, we have an exact sequence sheaves on $\sX_{\Nis}$:
\[ 0\to F_{(\sX,D')} \to F_{(\sX,D)} \to \iota_*(\Fcontt {e_1})_{(\sZ,D_\sZ)} \to 0,\]
which gives rise to a long exact sequence of sheaves on $\sH_{\Nis}$:
\begin{equation}\label{eq2;lem;purityM}
 \cdots \to R^\nu i_\sH^!F_{(\sX,D')}\to R^\nu i_\sH^! F_{(\sX,D)}\to
R^\nu i_\sH^!\iota_*(\Fcontt {e_1})_{(\sZ,D_\sZ)}\to\cdots.
\end{equation}
By the induction hypothesis, $R^\nu i_\sH^!F_{(\sX,D')}=0$ for $\nu>1$.
In case $\sH\not=\sZ$, we have a Cartesian diagram of closed immersions
\[\xymatrix{
\sH\cap \sZ \ar[r]^{\iota'} \ar[d]_{i_{\sH\cap \sZ}} & \sH \ar[d]^{i_\sH} \\
\sZ \ar[r]^{\iota}  & \sX \\
}\]
and we have an isomorphism
\[ R^\nu i_\sH^!\iota_*(\Fcontt {e_1})_{(\sZ,D_\sZ)} \simeq 
\iota'_* R^\nu i_{\sH\cap \sZ}^!(\Fcontt {e_1})_{(\sZ,D_\sZ)}.\]
By the induction hypothesis, $R^\nu i_{\sH\cap \sZ}^!(\Fcontt {e_1})_{(\sZ,D_\sZ)}=0$ for $\nu>1$ noting $\Fcontt {e_1}\in \CItspNis$ by Lemma \ref{lem:F-1splitting}.
So the desired vanishing follows from \eqref{eq2;lem;purityM}.
Moreover, the assumptions $(\spadesuit)$ and $\sH\not=\sZ$ imply that $\sH\not\subset |D|$. Then \eqref{eq1;lem;purityM} (with $q=1$) follows from \cite[Lem. 7.1(2)]{shuji}.

In case $\sZ=\sH$, we have
\[R^\nu i_\sH^!\iota_*(\Fcontt {e_1})_{(\sZ,D_\sZ)}=
R^\nu \iota^!\iota_*(\Fcontt {e_1})_{(\sZ,D_\sZ)},\]
which vanishes for $\nu>0$. Hence \eqref{eq2;lem;purityM} gives the desired vanishing together with an exact sequence:
\[ 0\to (\Fcontt {e_1})_{(\sH,D_\sH)} \rmapo {\delta} R^1 i_\sH^!F_{(\sX,D')}\to 
R^1 i_\sH^! F_{(\sX,D)}\to 0.\]
By \cite[Lem. 7.1(2)]{shuji} we have an isomorphism
\[ (\Fc)_{(\sH,D_\sH)} \simeq R^1 i_\sH^!F_{(\sX,D')}\]
through which $\delta$ is identified with the map induced by the canonical map
$\Fcontt {e_1} \to \Fc$.
This proves the desired isomorphism \eqref{eq1;lem;purityM} in case $\sZ=\sH$
and completes Step 1.

\medbreak\noindent
{\bf Step 2:}\;
We prove the theoerm by the induction on $q$ assuming $q>0$.
Let $I=\{i_1,\dots,i_q\}\subset [1,n]$ and $\sY\subset \sX$ be the closed subscheme defined by $\{t_{i_1}=0\}$. 
Let $i_\sY:\sY\hookrightarrow \sX$ and $i_{\sH,\sY}:\sH\to \sY$
be the induced closed immersions. By Step 1 we have
$R^\nu i_{\sY}^! F_{(\sX,D)} =0$ for $\nu\not=1$ and we have an isomorphism
\[(\tauu {e_{i_1}} F)_{(\sY,D_\sY)} \simeq R^1 i_\sY^! F_{(\sX,D)}\qwith
D_\sY=\{t_1^{e_1}\cdots \overset{\vee}{t_{i_1}^{e_{i_1}}}\cdots t_n^{e_n}=0\}\subset \sY.\]
Note $\tauu {e_{i_1}} F\in \CItspNis$ by Lemma \ref{lem:F-1splitting}.
Thus, by the induction hypothesis, we have
$R^\nu i_{\sH,\sY}^! \tauu {e_{i_1}} F_{(\sY,D_\sY)} =0$ for $\nu\not=q-1$.
By the spectral sequence
\[ E_2^{a,b}= R^{b} i_{\sH,\sY}^! R^a i_{\sY}^! F_{(\sX,D)} \Rightarrow
R^{a+b} i_{\sH}^! F_{(\sX,D)} ,\]
we get the desired vanishing \eqref{eq;lem;purityM} and an isomorphism 
\begin{multline*}
R^q i_{\sH}^! F_{(\sX,D)} \simeq  R^{q-1} i_{\sH,\sY}^! R^1 i_{\sY}^! F_{(\sX,D)}  \simeq R^{q-1} i_{\sH,\sY}^! (\tauu {e_{i_1}} F)_{(\sY,D_\sY)}\\
 \simeq (\tauu {e_{i_2},\dots,e_{i_q}}(\tauu {e_{i_1}} F))_{(\sH,D_\sH)}\simeq
 (\tauu {e_{i_1}, e_{i_2},\dots,e_{i_q}} F)_{(\sH,D_\sH)},
\end{multline*}
where the third isomorphism holds by the induction hypothesis.
This completes the proof of the theorem.
\end{proof}

We say $\sX=(X,D) \in \ulMCor$ reduced if so is $D$.
The following corollaries \ref{cor;purityM} and \ref{cor:purityreducedmodulus} are immediate consequences of Theorem \ref{thm;purityM}.

\begin{cor}\label{cor;purityM}
Take $F\in\CItspNis$ and $(X,D)\in \ulMCorls$ reduced.
Let $x\in \cd X n$ with $K=k(x)$ and let
$\sX=\hen{X}{x}$ be the henselization of $X$ at $x$. Then
\[H^i_x(X_{\Nis},F_{(X,D)})=0 \qfor i\not=n.\]
Choosing an isomorphism
\[\epsilon :\sX \simeq  \Spec K\{t_1,\dots,t_n\}\]
such that $D_{|\sX} = \{t_1^{e_1}\cdots t_n^{e_n}=0\}\subset \sX$ with $e_1,\dots,e_n\in \{0,1\}$, there exists an isomorphism depending on $\epsilon$:
\[ \theta_\epsilon: \tauu {e_1,e_2,\dots,e_n}F(x) 
\simeq H^n_x(X_{\Nis},F_{(X,D)}).\]
\end{cor}
\medbreak

\begin{cor}\label{cor:purityreducedmodulus}
For $F\in \CItspNis$ and $\sX=(X,D) \in \ulMCorls$ reduced,
the following sequence is exact:
\[ 0\to F(X,D) \to F(X-D,\emptyset) \to 
\underset{\xi\in \cd D 0}{\bigoplus} \frac{F(\hen X \xi-\xi,\emptyset)}{F(\hen X \xi,\xi)}.\]
\end{cor}

The idea of deducing the following corollary from the above is due to A. Merici.

\begin{cor}\label{cor2:purityreducedmodulus}
Let $\sX=(X,D)\in \ulMCorls$ be reduced. 
\begin{itemize}
\item[(1)]
Assume given an exact sequence in $\ulMNST$:
\eq{eq0;cor2:purityreducedmodulus}{ 0\to H \rmapo{\phi} G \rmapo{\psi} F}
such that $F,G, H\in \CItspNis$ and that $\ulomega_!\psi$ is surjective in $\NST$.
If $X$ is henselian local, 
\[ 0\to H(\sX) \to G(\sX) \to F(\sX) \to 0\]
is exact.
\item[(2)]
Let $\gamma: F\to G$ be a map in $\CItspNis$ such that $\ulomega_!\gamma$ is an isomorphism. Then $F(\sX)\to G(\sX)$ is an isomorphism. 
\item[(3)]
For $F\in \CItspNis$, the unit map $u: F\to \ulomegaCI\ulomega_! F$ induces an isomorphism $F(\sX)\cong \ulomegaCI\ulomega_! F(\sX)$. 
\end{itemize}
\end{cor}
\begin{proof}
To show (1), it suffices to show the surjectivity of $G(\sX) \to F(\sX)$.
Let $\eta\in X$ be the generic point and consider the following commutative diagram of the Cousin complexes:
\[\xymatrix{
0\ar[r] & H(\sX) \ar[r]\ar[d] & H(\eta) \ar[r]\ar[d]^{\phi(\eta)}&
\underset{x\in \cd X 1}{\bigoplus} H^1_x(X,H_\sX) \ar[r]\ar[d]^{H^1_x(\phi)}&
\underset{y\in \cd X 2}{\bigoplus} H^2_y(X,H_\sX) \ar[d]^{H^2_y(\phi)} \\
0\ar[r] & G(\sX) \ar[r]\ar[d]  & G(\eta) \ar[r]\ar[d]^{\psi(\eta)}&
\underset{x\in \cd X 1}{\bigoplus} H^1_x(X,G_\sX) \ar[r]\ar[d]^{H^1_x(\psi)}&
\underset{y\in \cd X 2}{\bigoplus} H^2_y(X,G_\sX) \ar[d]^{H^2_y(\psi)} \\
0\ar[r] & F(\sX) \ar[r] & F(\eta) \ar[r]&
\underset{x\in \cd X 1}{\bigoplus} H^1_x(X,F_\sX) \ar[r]&
\underset{y\in \cd X 2}{\bigoplus} H^2_y(X,F_\sX)  \\
}\]
By Corollary \ref{cor;purityM}, the horizontal sequences are exact.
By the assumption, $\psi(\eta)$ is surjective. 
By a diagram chase we are reduced to showing the following.

\begin{claim}
\begin{itemize}
\item[(i)]
For $x\in \cd X 1$, the sequence
\[ H^1_x(X,H_\sX) \to H^1_x(X,G_\sX) \to H^1_x(X,F_\sX) \]
is exact.
\item[(ii)]
For $y\in \cd X 2$, $H^2_y(\phi)$ is injective.
\end{itemize}
\end{claim}

To show (i), by Corollary \ref{cor;purityM}, it suffices to show the exactness of 
$\tauu e H \to \tauu e G \to \tauu e F$ for $e\in \{0,1\}$.
The case $e=0$ follows from the left exactness of the endofunctor 
$\uHom_{\uMPST}(\sX, -)$ on $\uMNST$ for any $\sX\in \ulMCor$.
We have a commutative diagram
\[\xymatrix{
\tauu 1 H \ar[r]^{\phi} \ar[d]^{s_H} & \tauu 1 G \ar[r]^{\psi} \ar[d]^{s_G}  & \tauu 1 F \ar[d]^{s_F}  \\
\tauu 0 H \ar[r]^{\phi}  \ar@<0.8ex>[u]^{p_H} & \tauu 0 G \ar[r]^{\psi}  
\ar@<0.8ex>[u]^{p_G} & \tauu 0 F \ar@<0.8ex>[u]^{p_F} \\}\]
where $p_*$ are the projections and $s_*$ is a right inverse of $p_*$ coming from the retractions from Lemma \ref{lem:F-1splitting}. We have
\[ \phi\circ p_H= p_G\circ \phi,\; \psi\circ p_G= p_F\circ \psi,
\; \phi\circ s_H= s_G\circ \phi,\; \psi\circ s_G= s_F\circ \psi.\]
By a diagram chase, the case $e=1$ follows from the case $e=0$.

To show (ii), by Corollary \ref{cor;purityM}, it suffices to show the injectivity of 
$\tauu {\ul{e}} H \to \tauu {\ul{e}}  G $ for $\ul{e}\in \{(0,0),(0,1),(1,0),(1.1)\}$.
The case $\ul{e}=(0,0)$ follows from the same left exactness as above, and 
the other cases from this case thanks to Lemma \ref{lem:F-1splitting}.

To show (2), we may assume $\sX$ is henselian local. Then it follows from (1).
(3) follows from (2) since $\ulomega_! u$ is an isomorphism.
This completes the proof of the corollary.
\end{proof}

\section{Review on higher local symbols}\label{HLS}

In this section we recall from \cite{rs-RSCpairing} the higher local symbols for reciprocity sheaves, which is a fundamental tool to prove Theorem \ref{thm:CItsp-log}, one of the main theorems of this paper.

\begin{para}\label{par:sp-ch}
Let $X$ be an excellent separate scheme of pure dimension $d$. 
We denote by $X^{(c)}$ the set of points of codimension $c$.
Let $x, y\in X$ be two points. We write
\[x> y :\Longleftrightarrow \ol{\{x\}}\supset \ol{\{y\}}, \text{ i.e. }, y\in \ol{\{x\}}.\]
A {\em specialization chain} (or just {\em chain}) in $X$ is a sequence
\[x=(x_1,\ldots, x_n) \quad \text{with }x_1> x_2> \ldots > x_n.\]
We say a specialization chain $x$ is {\em maximal} or {\em a Parsin chain} if $n=d$ and $ x_i\in X^{(i)}$\footnote{
The notation is different from \cite{rs-RSCpairing} in which 
a specialization chain $x=(x_1,\ldots, x_n)$ satisfies $x_n> \ldots > x_1$ and 
$x$ is maximal if the closures of $ x_i$ are of dimension $i$.}.
We denote by $c(X)$ the set of specialization chains in $X$ and 
\[\mc(X):=\{x\in c(X)|\; \text{$x$ is maximal}\}.\]
If $d=0$, the only element in $\mc(X)$ is the empty chain.

We say a specialization chain $x=(x_1,\dots,x_n)$ is {\em a $Q$-chain with break at $r$} if $0\leq r \leq n=d-1$ and $x_i\in \cd X i$ for $i\leq r$ and $x_i\in \cd X {i+1}$ for $i>r$. We denote 
\[Q_r(X):=\{\text{$Q$-chains with break at $r$ in }X\}.\]
For $x\in Q_r(X)$, we denote by $B(x)$ the set of all $y\in X$ such that 
\[ x(y):=(x_1,\dots ,x_r,y,x_{r+1},\dots,x_{d-1})\in \mc(X).\]

For a specialization chain $x=(x_1,\ldots, x_n)$ in $X$, we let $\sO_{X,x}^h$ be the henselization of $X$ along $x$ as defined in \cite[Def. 1.6.2]{KS}: If $n=1$, $\sO_{X,x}^h$ is the henselization of $\sO_{X,x}$. If $n>1$, let $y=(x_2,\dots,x_n)$ and assume that we have defined $R=\sO_{X,y}^h$. Let $T$ be the set of all prime ideals of $R$ lying over $x_1$. Then we define
$\sO_{X,x}^h$ as the product of the henselization of $R$ at $\fp$ for all $\fp\in T$.
By definition, $\sO_{X,x}^h$ is a finite product of henselian local rings and we let $\fm_x\subset \sO^h_{X,x}$ denote the product of its maximal ideals and $\kX x$ denote its total fraction ring. 

For $x\in Q_r(X)$ and $x(y)$ with $y\in B(x)$ as above, we have a natural inclusion of rings
\eq{iotaxy}{\iota_y : \kX x \to \kX {x(y)}.}

\end{para}
\medbreak

In what follows in this section, we fix $F\in \RSC_\Nis$ and write $\tF=\ulomegaCI F\in \CItsp_\Nis$ (cf. \eqref{omegaCIadjointNis}). 
We also fix a function field $K$ over the base field $k$.
Let $X$ be an integral $K$-scheme of dimension $d$. 
Recall from \cite{rs-RSCpairing} that we have a collection of bilinear pairings
\eq{HLC}{\{ (-,-)_{X/K,x}: F(K(X))\otimes K^M_d(\kX x)\to F(K)\}_{x\in \mc(X)},}
where we recall $\kX x$ is a finite product of fields and $ K^M_n(\kX x)$ for an integer $n>0$ denotes the product of the Milnor $K$-groups of the fields. We also note that $\Oh X x$ is a finite product of henselian dvr's and let $K^M_n(\Oh X x)$ denote the product of   the Milnor $K$-groups of the dvr's.

For a local ring $R$ and an ideal $I\subset R$, let $K^M_n(R,I)\subset K^M_n(R)$ denote the subgroup generated by symbols
\[ \{1+a,b_1,\dots,b_{n-1}\} \qwith a\in I,\, b_i\in R^\times.\]

The following properties hold for all $a\in F(K(X))$:

\begin{enumerate}[label= (HS\arabic*)]

\item\label{HS1}
Let $X\hookrightarrow X'$ be an open immersion where $X'$ is an integral $K$-scheme of dimension $d$. Then we have
 $(a,\beta)_{X/K,x}=(a,\beta)_{X'/K,x}$ for all $\beta\in K^M_d(\kX x)$.

\item\label{HS2} Let $x=(x_1,\ldots, x_d)\in\mc(X)$ and $X_1\subset X$ be the closure of $x_1$ and $\delta=(x_2,\ldots, x_d)\in \mc(X_1)$.
Then for all $\beta\in K^M_d(\kX x)$
\[(a, \beta)_{X/K, x} =\begin{cases} 
      \beta\cdot \Tr_{E/K}(a), & \text{if $d=0$ and $E=K(X)$};\\
     (a_{|X_1}, \partial_{x_1}\beta)_{X_1/K, \delta}, & \text{if } d\ge 1 \text{ and } 
a\in F(\sO_{X, x_1});
\end{cases}\]
where 
$a_{|X_1}\in F(K(X_1))$ is the restriction of $a$, and 
\[\partial_{x_1}: K^M_d(\kX x) \to K^M_{d-1}(\kh {X_1} \delta)\]
is the tame symbol coming from  the fact that $\sO_{X,x}^h$ is a product of dvr's and $\kh {X_1} \delta$ is identified with the product of its residue fields.

\item\label{HS3} Let $D\subset X$ be an effective Cartier divisor such that
$(X,D)\in \ulMCor^{\rm pro}$ (cf. \S\ref{preliminaries}\eqref{ulMCorpro} ).
Let $I_D\subset \sO_X$ be the ideal sheaf of $D$.
Assume $a\in \tF(X,D)$. Then, for all $x\in \mc(X)$, we have
\[(a, \beta)_{X/K,x}=0\qfor\forall  \beta\in  K^M_d(\sO_{X,x}^h, I_D \sO_{X,x}^h).\] 

\item\label{HS4} 
Let $x\in Q_r(X)$ with $0\leq r\leq d-1$. Then
$(a, \iota_y(\beta))_{X/K,  x(y)}=0$ for almost all $y\in B(x)$.
If $r<d-1$ or $X$ is projective over $K$, 
\[\sum_{y\in B(x)} (a, \iota_y(\beta))_{X/K,  x(y)}=0, 
   \quad \text{for all }\beta\in K^M_d(\kX x),\]
where $\iota_y: \kX x \to \kX {x(y)}$ is the natural map.

\item\label{HS5} 
Let $f: Y\to X$ be a finite surjective $K$-morphism between two integral $K$-schemes. Let $x\in \mc(X)$ and $y\in\mc(Y)$ with $f(y)=x$.
Then we have
\[(f^*a, \beta)_{Y/K,y}=(a, \Nm_{y/x}(\beta))_{X/K,x} \;(a\in F(K(X)),\; \beta\in K^M_d(\kY y)),\]
\[(f_*b,\alpha)_{X/K,x }=  \underset{\substack{z\in \mc(X')\\  f(z)=x}}{\sum}  (b,i_{x/z}\alpha)_{Y/K,z}\; (b\in F(K(X')),\; \alpha\in K^M_d(\kX x)),\]
where $\Nm_{y/x}: K^M_d(\kY y) \to K^M_d(\kX x)$ is the norm map on Milnor $K$-theory and $i_{x/z}: K^M_d(\kX x) \to K^M_d(\kY z)$ is the inclusion map.
\end{enumerate}

\medbreak

Now assume that $X\in \Sm$ of dimension $d$ and $D\subset X$ be a reduced SNCD on $X$. For a function field $K$ over $k$ and a scheme $Z$ over $k$, write $Z_K=Z\otimes_k K$ with $\phi_Z:Z_K\to Z$ the projection.
If $Z$ is integral, we denote by $K(Z)$ the function field of $Z_K$.
We quote the following result from \cite[Pr. 8.4]{rs-RSCpairing}.

\begin{prop}\label{lem2;diagonal}
Let $F\in \RSC_\Nis$ and $a\in F(X-|D|)$. 
Assume that there exists an open subset $U\subset X$ which contains all generic points of $D$ such that the following condition holds:
For any function field $K$ over $k$ and any $x=(x_1,\dots,x_d)\in \mc(U_K)$ with $x_1\in D_K^{(0)}$, we have
\[ ( \phi_X^*(a),\beta)_{U_K/K,x}=0\qfor \forall \beta\in K_d^M(\sO_{X_K,x}^h,\fm_x).\]
Then we have $a\in F(X,D)$.
\end{prop}


\section{Logarithmic cohomology of reciprocity sheaves}\label{logcoh}

For $\sX=(X,D)\in \ulMCorls$, we write $\sX_{\red}=(X,|D|)\in \ulMCorls$.
We say $\sX=(X,D)\in \ulMCorls$ is reduced if $\sX=\sX_{\red}$.

\begin{defn}\label{def:CItsp-log}
Let $F\in \ulMPST$.
\begin{itemize}
\item[(1)]
We say that $F$ is \emph{log-semipure} if for any $\sX\in \ulMCorls$,
the map $F(\sX_{\red}) \to F(\sX)$ is injective.
Note that if $F$ is semipure, $F$ is log-semipure.
\item[(2)]
We say that $F$ is \emph{logarithmic} if it is log-semipure and satisfies the condition that for $\sX,\sY\in \ulMCorls$ with $\sX$ reduced and 
$\alpha\in \ulMCorfin(\sY,\sX)$, the image of 
$\alpha^*: F(\sX) \to F(\sY)$
is contained in $F(\sYred)\subset F(\sY)$, where $\sYred=(Y,E_{\red})$ for $\sY=(Y,E)$.
\end{itemize}
Let $\uMPSTlog$ be the full subcategory of $\ulMPST$ consisting of logarithmic objects and put $\uMNSTlog=\uMNST\cap \uMPSTlog$.
\end{defn}

\begin{thm}\label{thm:CItsp-log}
Any $F\in \CItspNis$ is logarithmic, i.e. $\CItspNis\subset \uMNSTlog$.
\end{thm}

We need a preliminary for the proof of the theorem.

\begin{lemma}\label{lem:contraction}
Let $F\in \CItspNis$.
Let $\A^n_K=\Spec K[x_1,\dots,x_n]$ be the affine space over a function field $K$ over $k$ and $\sX=\Spec K\{x_1,\dots,x_n\}$ be the henselization of $\A^n_K$ at the origin.
Let $L_i=\{x_i=0\}\subset \A^n$ and $\sL_i=L_i\times_{\A^n_K} X$ for $i\in [1,n]$.
For an integer $0<r\leq n$, the natural map 
\[ K\{x_{r+1},\dots,x_{n}\} [x_1,\dots,x_r] \to K\{x_1,\dots,x_n\}\]
induces a map in $\ulMCorls$:
 \[\rho_r:
(\sX,\sL_1+\cdots +\sL_r) \to (\A^r_S,\{x_1\cdots x_r=0\})\simeq
(\A^1,0)^{\otimes r}\otimes (S,\emptyset) ,\]
where $S=\Spec K\{x_{r+1},\dots,x_{n}\}$. It induces 
\begin{equation}\label{eq1;lem:contraction}
\rho_r^*:  F(\A^r_S,\{x_1\cdots x_r=0\} ) \to  F(\sX,\sL_1+\cdots + \sL_r)
 \end{equation}
Then $F(\sX,\sL_1+\cdots + \sL_r) $ is generated by the image of $\rho_r^*$ and 
\[F(\sX,\sL_1+\cdots+\overset{\vee}{\sL_i}+\cdots \sL_{r}) \qfor i=1,\dots, r.\]
\end{lemma}
\begin{proof}
For $\sY\in \ulMCor$, let $F^\sY\in \ulMPST$ be defined by 
$F^\sY(\sZ)=F(\sY\otimes\sZ)$. By \cite[Lem. 1.5(3)]{brs}, $F^\sY\in \CItspNis$ for
$F\in \CItspNis$. 
We prove the lemma by the induction on $r$. The case $r=1$ holds since by \cite[Lem. 7.1 and Lem 5.9]{shuji}, $\rho_1$ induces an isomorphism 
\[ F^{(\A^1,0)}(S)/ F^{(\A^1,\emptyset)}(S) \isom F(\sX,\sL_1)/F(\sX).\]
By definition $\sL_1=\Spec K\{x_2,\dots,x_n\}$ and we have a map in $\ulMCor$:
 \[(\sX,\sL_1+\cdots +\sL_r) \to (\A^1,0)\otimes(\sL_1,\sL_1\cap (\sL_2+\cdots +\sL_r))\]
induced by the natural map $K\{x_2,\dots,x_{n}\} [x_1] \to K\{x_1,\dots,x_n\}$. 
By \cite[Lem. 7.1 and Lem 5.9]{shuji}, it induces an isomorphism
\[ F^{(\A^1,0)}(\sL_1,E)/ F^{(\A^1,\emptyset)}(\sL_1,E) \isom
 F(\sX,\sL_1+\cdots +\sL_r)/F(\sX,\sL_2+\cdots +\sL_r)\] 
with $E=\sL_1\cap (\sL_2+\cdots +\sL_r)$. By the induction hypothesis,
$F^{(\A^1,0)}(\sL_1,E)$ is generated by $F^{(\A^1,0)}(\sL_1,E_j)$ with
$E_j=\sL_1\cap(\sL_2\cdots+\overset{\vee}{\sL_j}+\cdots \sL_{r})$ for $j=2,\dots, r$
together with the image of the map
\[ (F^{(\A^1,0)})^{(\A^1,0)^{\otimes r-1} }(S) =F^{ (\A^1,0)^{\otimes r} }(S) \to F^{(\A^1,0)}(\sL_1,E)\]
induced by 
 \[(\sL_1,E) \to (\A^{r-1}_{S},\{x_2\cdots x_r=0\})\simeq(\A^1,0)^{\otimes r-1}\otimes (S,\emptyset) \]
coming from the map $K\{x_{r+1},\dots,x_n\}[x_2,\dots,x_r] \to K\{x_2,\dots,x_d\}$. This proves the lemma.
\end{proof}

\def\alphab{\overline{\alpha}}
\medbreak\noindent{\it Proof of Theorem \ref{thm:CItsp-log}} :
By Corollary \ref{cor2:purityreducedmodulus}(3), we may assume $F=\ulomegaCI G$ for $G\in \RSCNis$. Take $\sX=(X,D),\sY=(Y,E)\in \ulMCorls$ with $\sX$ reduced and let $\alpha\in \ulMCorfin(\sY,\sX)$ be an elementary correspondence. We need to show that $\alpha^*(F(\sX))\subset F(\sY_{\red})$.
The question is Nisnevich local over $X$ and $Y$.
Hence we may assume $(X,D)=(\sX,\sL_1+\cdots +\sL_r)\in \uMCor^{\rm pro}$ under the notation from Lemma \ref{lem:contraction}. If $r=0$, we have $\alpha\in \ulMCor((Y,\emptyset),(X,\emptyset))$ by the assumption  $\alpha\in  \ulMCorfin(\sY,\sX)$ so that
\[\alpha^*(F(\sX))=\alpha^*(F(X,\emptyset)) \subset F(Y,\emptyset)\subset F(\sY_{\red}).\]
Assume $r>0$ and proceed by the induction on $r$. By Lemma \ref{lem:contraction},
we may assume then 
\[(X,D)= \sM:=(\A^1,0)^{\otimes r}\otimes(S,\emptyset) \qfor S\in \Sm.\]
On the other hand, by Corollary \ref{cor:purityreducedmodulus}, we have an exact sequence
\[ 0\to F(Y,E_{\red}) \to F(Y-E_{\red},\emptyset) \to 
\underset{\xi\in \cd E0}{\bigoplus} \frac{F(\hen Y \xi-\xi,\emptyset)}{F(\hen Y \xi,\xi)}.\]
Hence we may replace $Y$ with its Nisnevich neighborhood of a generic point $\xi$ of $E$. Using the assumption that $k$ is perfect, we may then assume the following condition $(\spadesuit)$. Recall that $\alpha$ is by definition an integral closed subschem of $(Y-E)\times (X-D)$ finite surjective over $Y-E$ and its closure $\alphab$ in $Y\times X$ is finite surjective over $Y$.
\begin{enumerate}
\item[$(\spadesuit)$]
$X$ and $Y$ are irreducible and $\alpha$ is smooth.
The normalization $Y'$ of $\alphab$ is smooth and $E':=E\times_Y Y'\subset Y'$ is irreducible and $E'_{\red}$ is smooth.
\end{enumerate}
Let $g: Y'\to Y \qaq f: Y'\to X$ be the induced maps. We have
$E'=g^*E \geq f^*D$ as Cartier divisors on $Y'$ by the modulus condition for $\alpha$. Hence these maps induce
\[ F(X,D)\rmapo{f^*} F(Y',E') \rmapo{g_*} F(Y,E).\]
We claim that 
$\alpha^*:F(X,D)\to F(Y,E)$ agrees with this map.
Indeed, this follows from the equality 
\[ \Gamma_f \circ ^t\Gamma_g= \alpha \in \Cor(Y-E,X-D),\]
where
$^t\Gamma_g\in \Cor(Y-E,Y'-E')$ is the transpose of the graph of $g$ and
$\Gamma_f\in \Cor(Y'-E',X-D)$ is the graph of $f$.
By definition this follows from the equality
\[^t\Gamma_g\times_{Y'-E'} \Gamma_f =\alpha\subset (Y-E)\times (X-D)\]
which one can check easily noting $Y'\to \alphab$ is an isomorphism over $\alpha$  assumed to be smooth.
Then we get a commutative diagram
\[\xymatrix{
&F(Y',E'_{\red}) \ar[d]^{\hookrightarrow}\\
&F(Y',E_{\red} \times_Y Y') \ar[r]^-{g_*} \ar[d]^{\hookrightarrow} 
& F(Y,E_{\red}) \ar[d]^{\hookrightarrow} \\
F(X,D)\ar[r]^{f^*} &F(Y',E') \ar[r]^-{g_*} & F(Y,E)  \\}\]
where the top inclusion comes from $E_{\red} \times_Y Y'\geq E'_{\red}$ as Cartier divisors on $Y'$ thanks to the sempurity of $F$ (cf. \S\ref{preliminaries}\eqref{semipure}). 
Hence it suffices to show $f^*(F(X,D))\subset F(Y',E'_{\red})$. By replacing $(Y,E)$ with $(Y',E')$, we may now assume that $\alpha$ is induced by
a morphism $f: Y\to X=\A^r\times S$.
Then $\alpha$ factors in $\ulMCor$ as 
\[(Y,E) \to (\A^1,0)^{\otimes r}\otimes (Y,\emptyset) \to
 (\A^1,0)^{\otimes r}\otimes(S,\emptyset).\]
where the first map is induced by the map
\[i=(pr_{\A^r}\circ f,id_Y): Y \to \A^r\times Y,\]
and the second induced by 
\[id_{\A^r}\times (pr_S\circ f): \A^r\times Y \to \A^r\times S.\]
Note that $i$ is a section of the projection $\A^r\times Y \to Y$.
Thus we are reduced to showing  
$i^*(F((\A^1,0)^{\otimes r}\otimes(Y,\emptyset))\subset F(Y,E_{\red})$. By Proposition \ref{lem2;diagonal} this follows from the following.
\def\zdel{z(\delta)}
\def\wdel{w(\delta)}

\begin{claim}
Assume $F=\ulomegaCI G$ for some $G\in \RSCNis$.
Take $a\in F((\A^1,0)^{\otimes r}\otimes(Y,\emptyset))$.
Let $K$ be a function field over $k$.
After replacing $Y$ by an open subset containing $\xi$, we have 
\[(i^*(a)_K,\gamma)_{Y_K/K,\delta}=0\qfor \forall \gamma\in K^M_e(\sO^h_{Y_K,\delta},\fm_\delta)\]
and for any $\delta=(\xi, \delta_1,\dots,\delta_{e-1})\in \mc(Y_K)$, where $e=\dim(Y)$ and $\xi\in E$ is the generic point, and
\[ (-,-)_{Y_K/K,\delta}: F(K(Y))\otimes K^M_d(\kh {Y_K} \delta)\to F(K)\]
is from \eqref{HLC}.
\end{claim}
\begin{proof}
Write
\[\A^r\times Y=\Spec A[x_1,\dots,x_r],\;
 (\A^1,0)^{\otimes r}\otimes(Y,\emptyset) =(\A^r_Y,\{x_1\cdots x_r=0\}).\]
After replacing $Y$ by an open subset containing $\xi$, we can write
\[ i(Y) =\underset{1\leq i\leq r}{\bigcap} \{x_i-u_i\pi^{m_i}=0\}
\qwith m_i\in \Z_{\geq 0},\; u_i\in A^\times.\]
Let $\delta=(\xi, \delta_1,\dots,\delta_{e-1})$ be as in the claim and put $\delta'=(\delta_1,\dots,\delta_{e-1})\in \mc((E_{\red})_K)$.
Let $X_K=\A^r\times Y_K$ and $z_j$ for $1\leq j\leq r$ be the generic point of 
\[ Z_j=\underset{1\leq i\leq j}{\bigcap} \{x_i-u_i\pi^{m_i}=0\} \subset X_K ,\]
and $w_j$ for $1\leq j\leq r$ be the generic point of 
\[W_j=\{\pi=x_1=\cdots=x_j=0\} = F\cap Z_j \qwith F=\{\pi=0\}\subset X_K .\]
The section $i$ induces isomorphisms
\eq{iZrWr}{Y_K\simeq Z_r \qaq 
 (E_{\red})_K\simeq W_r.}
\def\tsigma{\tilde{\sigma}}
Let $\sigma=(\eta_1,w_1,\dots,w_r,i(\delta'))\in \mc(X_K)$, where $\eta_1$ is the generic point of $D_1=\{x_1=0\}\subset X_K$ and $i(\delta')\in \mc(W_r)$ is the image of $\delta'$ under \eqref{iZrWr}. Let $\tsigma$ be the chain $(w_r,i(\delta'))$ in $X_K$. 
By the definition of henselization along chains, the projection $X=\A^r\times Y \to Y$ induces $\iota: \sO^h_{Y_K,\delta} \to \sO^h_{X_K,\tsigma}$. 
Take any $\gamma\in K^M_e(\sO^h_{Y_K,\delta},\fm_\delta)$ and put 
\begin{equation}\label{eq;beta}
 \beta=\{\iota(\gamma),\frac{u_1\pi^{m_1}-x_1}{u_1\pi^{m_1}},\dots,\frac{u_r\pi^{m_r}-x_r}{u_r\pi^{m_r}}\}\in K^M_d(K^h_{X_K,\tsigma})\;(d=e+r),\end{equation}
For $a\in F((\A^1,0)^{\otimes r}\otimes(Y,\emptyset))$ and its restriction 
$a_K\in F((\A^1,0)^{\otimes r}\otimes(Y_K,\emptyset))$, we have\footnote{By abuse of notation, for $\rho\in \mc(X_K)$ containing $\tsigma$ as a sub-chain,  
$\beta$ in $(a_K,\beta)_{X_K/K,\rho}$ denotes the image of $\beta$ from \eqref{eq;beta} under the natural map 
$K^M_d(K^h_{X_K,\tsigma})\to K^M_d(K^h_{X_K,\rho})$. The same convention for 
$\beta_1$ and $\beta_2$ below.}
\[ \begin{aligned}
0=(a_K,\beta)_{X_K/K,\sigma}
&=-\underset{\substack{\tau\in X_K^{(1)}\\\tau>w_1,\tau\not=\eta_1}}{\sum}
(a_K,\beta)_{X_K/K,(\tau,w_1,\dots,w_r,i(\delta'))} \\
&=-(a_K,\beta)_{X_K/K,(z_1,w_1,\dots,w_r,i(\delta'))} \\
&=\pm ((a_K)_{|Z_1},\beta_1)_{Z_1/K,(w_1,\dots,w_r,i(\delta'))},
\\\end{aligned}\]
\[ \beta_1=\{\iota_1(\gamma),\frac{u_2\pi^{m_2}-x_2}{u_2\pi^{m_2}},\dots,\frac{u_r\pi^{m_r}-x_r}{u_r\pi^{m_r}}\}\in K^M_{d-1}(K^h_{Z_1,\tsigma_1}))\]
where $\tsigma_1$ is the chain $(w_r,i(\delta'))$ in $Z_1$ and 
$\iota_1: \sO^h_{Y_K,\delta} \to \sO^h_{Z_1,\tsigma_1}$ is induced by 
$Z_1\hookrightarrow X_K \to Y_K$.
The first equality follows from \S\ref{HLS} \ref{HS3} applied to $D_1\subset X_K$ noting that $\beta$ lies in $K^M_d(\sO^h_{X_K, \sigma},\fm_{\sigma})$ since 
$(u_1\pi^{m_1}-x_1)/u_1\pi^{m_1} \in 1+ x_1 \sO_{X_K,\eta_1}$. The second follows from \ref{HS4} applied to $x=(w_1,\dots.w_r,i(\delta'))\in Q_0(X_K)$.
The third equality holds since $z_1$ is the unique $\tau\in X_K^{(1)}-\{\eta_1\}$ such that $\tau>w_1$ and $(a_K,\beta)_{X_K/K,(\tau,w_1,\dots,w_r,i(\delta'))} $ may not vanish, which follows from \ref{HS2} noting $\iota(\gamma)_{|F}=0$. 
Finally the last equality follows from \ref{HS2}.
We further get
\[ \begin{aligned}
0=((a_K)_{|Z_1},\beta_1)_{Z_1/K,(w_1,w_2,\dots,w_r,i(\delta'))} 
&=-\underset{\substack{\tau\in Z_1^{(1)}\\\tau>w_2,\tau\not=w_1}}{\sum}
((a_K)_{|Z_1},\beta_1)_{Z_1/K,(\tau,w_2,\dots,w_r,i(\delta'))}  \\
&=-((a_K)_{|Z_1},\beta_1)_{Z_1/K,(z_2,w_2,\dots,w_r,i(\delta'))}  \\
&=\pm ((a_K)_{|Z_2},\beta_2)_{Z_2/K,(w_2,\dots,w_r,i(\delta'))} ,\\\end{aligned}\]
\[ \beta_2=\{\iota_2(\gamma),\frac{u_3\pi^{m_3}-x_3}{u_3\pi^{m_3}},\dots,\frac{u_r\pi^{m_r}-x_r}{u_r\pi^{m_r}}\}\in K^M_{d-1}(K^h_{Z_2,\tsigma_2})),\]
where $\tsigma_2$ is the chain $(w_r,i(\delta'))$ in $Z_2$ and 
$\iota_2: \sO^h_{Y_K,\delta} \to \sO^h_{Z_2,\tsigma_2}$ is induced by 
$Z_2\hookrightarrow X_K \to Y_K$.
The above equalities hold by the same arguments as above except that 
for the third equality, there are a priori two $\tau\in Z_1^{(1)}-\{w_1\}$ with $\tau>w_2$ for which $((a_K)_{|Z_1},\beta_1)_{Z_1/K,(\tau,w_2,\dots,w_r,i(\delta'))}$ may not vanish. One is $z_2$ and another is the generic point $\eta_2$ of $Z_1\cap D_2$ with $D_2= \{x_2=0\}\subset X_K$, but 
$((a_K)_{|Z_1},\beta_1)_{Z_1/K,(\eta_2,w_2,\dots,w_r,i(\delta'))}=0$.
Indeed, 
$(a_K)_{|Z_1}\in F(\Spec(\sO_{Z_1,\eta_2}),\eta_2)$ since $Z_1$ and $D_2$ intersect transversally in $X_K$. Hence the vanishing follows from \ref{HS3} applied to $Z_1\cap D_2\subset Z_1$ noting $\beta_1\in K^M_d(\sO_{Z_1,(\eta_2,\tsigma_1)},\fm_{\eta_2})$.
Repeating the same arguments, we finally get
\[ 0=((a_K)_{|Z_r},\iota_r(\gamma))_{Z_r/K,(w_r,i(\delta'))} =((a_K)_{|Y_K},\gamma)_{Y_K/K,\delta},\]
where $\iota_r: \sO^h_{Y_K,\delta} \to \sO^h_{Z_r,(w_r,i(\delta'))}$ is induced by $Z_r\hookrightarrow X_K \to Y_K$ and the second equality follows from \eqref{iZrWr}.
This completes the proof of the claim and Theorem \ref{thm:CItsp-log}.
\end{proof}


\begin{defn}\label{def:logcoh}
For $F\in\uMNSTlog$ and an integer $i\geq 0$, consider the association
\[ \Hlog i(-,F) : \uMCorfinls \to \Ab\;;\; (X,D) \to H^i(X_\Nis, F_{(X,D_{\red})}).\]
By the definition this gives a presheaf on $\uMCorfinls$, which we call  
\emph{the $i$-th logarithmic cohomology with coefficient $F$}.
\end{defn}

\section{Invariance of logarithmic cohomology under blowups}\label{invariance-bu}
\def\Lfinls{\Lambda^{\fin}_{ls}}

Let the notation be as in \S\ref{logcoh}.
Let $\Lfinls$ be the subcategory of $\uMCorfinls$ whose objects are the same as $\uMCorfinls$ and whose morphisms are those $\rho: (Y,E) \to (X,D)$ where $E=\rho^* D$ and $\rho$ are induced by blowups of $X$ in smooth centers $Z\subset D$ which are normal crossing to $D$ in the following sense: For any point $x$ of $D$, there exists a system $z_1, \cdots, z_d$ of regular parameters of $X$ at $x$ 
satisfying the conditions:
\begin{itemize}
\item Locally at $x$, $Z =\{z_1= \dots = z_r=0\}$ with $r=\codim_{X} Z$.
\item Locally at $x$, $|D| = \{\prod_{j \in J} z_j=0 \}$ 
for some $J\subset \{1,\dots,r\}$.
\end{itemize}

\begin{thm}\label{thm:CItsp-loginv}
For $F\in \CItspNis$ and $\rho:\sY\to \sX$ in $\Lfinls$, we have
\begin{equation}\label{eq1;thm:CItsp-loginv}
 \rho^*: \Hlog i(\sX,F) \cong \Hlog i(\sY,F) \qfor \forall i\geq 0.\end{equation}
\end{thm}
\begin{proof}
Writing $\sY=(Y,E)$ and $\sX=(X,D)$, $\rho$ is induced by a blowup
$\rho:Y\to X$ in a smooth center $Z\subset D$ normal crossing to $D$.
First we prove the theorem in case $i=0$.
We may assume that $D$ is reduced and $E=\rho^*D$.
Then $\rho$ is invertible in $\uMCor$ so that $\rho^*:F(\sX)\cong F(\sY)$.
Since this factors through $F(Y,E_{\red})$, we get \eqref{eq1;thm:CItsp-loginv} for $i=0$. 

To show \eqref{eq1;thm:CItsp-loginv} for $i>0$, it now suffices to prove
$R^i\rho_* F_{(Y,E_{\red})} =0$.
By \cite[Lem. 9]{kes}, Nisnevich locally around a point of $Z$,
$(X,D)$ is isomorphic to 
\[ (\A^c,L_1+\cdots +L_r)\otimes \sW \qwith 
\sW=(W,W^\infty)\in \uMCorls,\]
where $\A^c=\Spec k[t_1,\dots,t_c]$ with $c=\codim_z(Z,X)$ and $L_i=V(t_i)$ for $i=1,\dots,r$ with $1\leq r\leq c$, and $Z$ corresponds to $0\times W$. Hence the theorem follows from the following proposition.
\end{proof}

\begin{prop}\label{prop:blowup-logAn}
Let $F\in\CItspNis$ and $\sW=(W,W^\infty)\in \ulMCorls$.
Let $\A^n=\Spec k[t_1,\dots,t_n]$ and put $L_i=V(t_i)$ for $1\leq i\leq n$.
Let $\rho: Y\to \A^n$ be the blow-up at the origin $0\in \A^n$ and
$\tL_i\subset Y$ be the strict transforms of $L_i$ for $i=1,\dots, r$ and $E=\rho^{-1}(0)\subset Y$.  
For any $1\leq r\leq n$, we have
\begin{equation}\label{eq1;blowup-logAn}
R^i\rho_{W*} F_{(Y, \tL_1 +\cdots +\tL_r  + E)\otimes \sW}=0\qfor i\ge 1,
\end{equation}
where  $\rho_W:=\rho\times\id_W: Y\times W\to \A^2\times W$.
\end{prop}

\begin{lemma}\label{lem:blowup-logA2}
Proposition \ref{prop:blowup-logAn} holds for $n=2$.
\end{lemma}
\begin{proof}
We can assume $W$ is henselian local.
The case $r=1$ is proved in \cite[Lem. 2.13]{brs} and we show the case $r=2$.\footnote{The following argument is adopted from \cite[Lem. 2.13]{brs}, but the present case is easier.} 
Put $D= L_1+ L_2$.
By the case $i=0$ of Theorem \ref{thm:CItsp-loginv}, we get
\begin{equation}\label{eq1;blowup-logA2}
F_{(\A^2,D)\otimes \sW} \cong \rho_{W*} F_{(Y, \tL_1+\tL_2 + E)\otimes \sW}.
\end{equation}
Set
\[\sF:=F_{(Y, \tL_1+\tL_2 + E)\otimes \sW}.\]
Since $R^i\rho_{W*}\sF$ for $i\ge 1$ is supported in $0\times W$ we have 
\[R^i\rho_{W*}\sF=0\Longleftrightarrow H^0(\A^2_W, R^i\rho_{W*}\sF)=0,\]
where $\A^2_W=\A^2\times W$, and 
\[H^j(\A^2_W, R^i\rho_{W*}\sF)=0, \quad \text{for all }i,j\ge 1.\]
By \eqref{eq1;blowup-logA2} and \cite[Lem. 2.11]{brs}
\[H^i(\A^2_W, \rho_{W*}\sF)=H^i(\A^2_W, F_{(\A^2, D)\otimes \sW})=0.\]
Thus the Leray spectral sequence yields
\[H^0(\A_W^2, R^i\rho_{W*}\sF)= H^i(Y\times W, \sF), \quad i\ge 0,\]
and we have to show, that this group vanishes for $i\ge 1$.
We can write 
\[\A^2=\Spec k[x,y]\;\text{ and }\;  L_1=V(x),\; L_2=V(y) \subset \A^2.\]
Then we have
\[Y=\Proj k[x,y][S,T]/(xT-yS)\subset \A^2\times \P^1.\]
Denote by 
\[\pi_0: Y\inj \A^2\times \P^1 \to \P^1=\Proj k[S,T]\]
the morphism induced by projection and let
$\pi: Y\times W \to \P^1_W$ be its base change. Then $\pi_0$ induces
an isomorphism $E\simeq \P^1$, and we have
\begin{equation}\label{eq;sUV0} \tL_1=\pi_0^{-1}(0),\;\; \tL_2=\pi_0^{-1}(\infty).\end{equation}
Set $s=S/T=x/y$  and write
\[ \P^1\setminus \{\infty\}= \A^1_s:= \Spec k[s],\quad \P^1\setminus\{0\}= \Spec k[\tfrac{1}{s}].\]
Set $U:=\A^1_s \times W$ and $V:=(\P^1\setminus\{0\})\times W$ and 
\[\sU:=(\A^1_s, 0)\otimes\sW, \quad \sV:=(\P^1\setminus\{0\},\infty) \otimes \sW.\]
We have 
\begin{equation*}\label{eq;sUV}
\pi^{-1}(U)= \A^1_y\times U, 
 \quad \pi^{-1}(V)=\A^1_x\times  V,
\end{equation*}
and the restriction of $\pi$ to these open subsets is given by projection.
Furthermore, $E\times W\subset Y$ is defined by $y=0$ on $\pi^{-1}(U)$ and 
by $x=0$ on $\pi^{-1}(V)$. In view of \eqref{eq;sUV0}, we have 
\begin{equation}\label{eq;sFUV}
\sF_{|\pi^{-1}(U)}=F_{(\A^1_y, 0)\otimes \sU},\quad 
\sF_{|\pi^{-1}(V)}= F_{(\A^1_x, 0)\otimes \sV}.
\end{equation}
Thus \cite[Lem. 2.11]{brs} yields
\[R^j\pi_*\sF=0 \qfor j\ge 1,\]
and it remains to show
\eq{lem:blow-upA21}{H^i(\P^1_W, \pi_*\sF)=0 \qfor i\ge 1.}
where $\P^1_W=\P^1\times W$. For this consider the map 
\[a_0: Y \to \A^1_x \times \P^1\]
which is the closed immersion $Y\inj \A^2\times \P^1$ followed by the projection $\A^2\to \A^1_x$. Let $a: Y\times W\to \A^1_x\times \P^1\times W$ be its base change. In view of \eqref{eq;sFUV}, the map $a$ induces a morphism in $\uMCor$:
\[\alpha: (Y,\tL_1+\tL_2+E)\otimes \sW\to (\A^1_x, 0)\otimes (\P^1,\infty)\otimes \sW,\]
which is an isomorphism over
$(\A^1_x, 0)\otimes (\P^1\backslash \{0\},\infty)\otimes \sW$. Setting
\[F_1:=\uHom(\Ztr(\A^1_x, 0), F )\in \CItspNis,\]
it induces a map of Nisnevich sheaves on $\P^1_W$:
\[ \pi_*(\alpha^*): F_{1, (\P^1,\infty)\otimes \sW}\to \pi_*\sF,\]
which becomes an isomorphism over $(\P^1-\{0\})\times W$.
Hence \eqref{lem:blow-upA21} follows from
\[H^i(\P^1_W, F_{1, (\P^1,\infty)\otimes \sW})=0 \qfor i\ge 1,\]
 which follows from \cite[Th. 0.6]{shuji}.
\end{proof}

\begin{lemma}\label{lem:blow-upInd}
Let $N>2$ be an integer and assume that Proposition \ref{prop:blowup-logAn} holds for $n<N$. Let $(X,D)\in \uMCorls$ and $Z\subset X$ be a smooth integral closed subscheme with $2\le \codim(Z,Y)=:c <N$. Assume
\[D= D_1+\cdots + D_r + D' \qwith r\leq c ,\]
where $D_1,\dots,D_r$ are distinct and reduced irreducible components of $D$ containing $Z$ and $D'$ is an effective divisor on $X$ such that none of the component of $D'$  contains $Z$ and $Z$ is transversal to $|D'|$.
Let $\rho: Y\to X$ be the blow-up of $X$ in $Z$ and $\tD_i,\tD'\subset Y$
be the strict transforms of $D_i$ and $D'$ respectively and $E_Z=\rho^{-1}(Z)$. 
Then, for all $\sW=(W,W^\infty)\in \uMCorls$, 
\[R^i\rho_{W*} F_{(Y, \tD_1+\cdots + \tD_r +E_Z + \tD' )\otimes \sW}=0 \qfor i\ge 1,\]
where $\rho_W: Y\times W \to X\times W$ denotes the base change of $\rho$.
\end{lemma}
\begin{proof}\footnote{The proof is adopted from \cite[Lem. 2.14]{brs}.}
The question is Nisnevich local around the points in $Z\times W$.
Let $z\in Z\times W$ be a point and set $A:=\sO_{X\times W,z}^h$.
For $V\subset Y\times W$ we denote by $V_{(z)}:= V\times_{X\times W} \Spec A$.
By assumption we find a system of local parameters $t_1,\ldots, t_m$ of $A$,
such that 
\[(D_i\times W)_{(z)}=V(t_i) \qfor 1\leq i\leq r,\;
(Z\times W)_{(z)}=V(t_1,\ldots, t_c),\]
\[ (D'\times W)_{(z)}=V(t_{c+1}^{e_{c+1}}\cdots t_{m_0}^{e_{m_0}}) \qwith c+1\leq m_0\le m, \]
\[(X\times W^\infty)_{(z)}=V(t_{m_0+1}^{e_{m_0+1}}\cdots t_{m_1}^{e_{m_1}})
\qwith m_0\le m_1\leq  m.\]
Letting $K$ be the residue field of $A$, we can choose a ring homomorphism $K\hookrightarrow A$ which is a section of $A \to K$. Then we obtain an isomorphism
\[K\{t_1,\ldots, t_m\}\xr{\simeq} A.\]
Let $\rho_1:\widetilde{\A^c}\to \A^c$ be the blow-up in $0$.
By the above 
\[\rho_W: (Y,\tD_1+\cdots + \tD_r +E_Z + \tD' )\otimes \sW\to (X, D)\otimes \sW\]
is Nisnevich locally around $z$
isomorphic over $k$ to the morphism
\[(\widetilde{\A^c},\tL_1+\cdots+\tL_r + E)\otimes \sW'
\to (\A^c, L_1+\cdots+ L_r)\otimes \sW',\]
\[ (\sW'=(\A_K^{m-c}, (\prod_{i=c+1}^{m_1} t_{i}^{e_{i}}))) \]
induced by a map $(\widetilde{\A^c},\tL_1+\cdots+\tL_r + E)\to (\A^c, L_1+\cdots+ L_r)$
as in  Proposition \ref{prop:blowup-logAn}.
Hence the statement follows from the proposition for $n=c<N$.
\end{proof}

\begin{proof}[Proof of Proposition {\ref{prop:blowup-logAn}}.]
The proof is by induction on $n\geq 2$. The case $n=2$ follows from Lemma \ref{lem:blowup-logA2}. 
Assume $n>2$ and the theorem is proven for $\A^m$ with $m<n$.
In case $r=1$, Proposition {\ref{prop:blowup-logAn}} is proved in \cite[Th. 2.12]{brs}.
Assume $r\geq 2$. 
Let $Z:=L_1\cap L_2\subset \A^n$ and $\tZ\subset Y$ be the strict transform of $Z$. Denote by $\rho': Y'\to Y$ the blow-up of $Y$ in $\tZ$ and $\tL_i', E'\subset Y'$ be the strict transforms of $\tL, E$ respectively and $E''=(\rho')^{-1}(\tZ)$. 
Note that $\tZ=\tL_1\cap \tL_2$ intersecting transversally with $\tL_3+\cdots+\tL_r+ E$ and $\codim(\tilde{Z}, Y)=2$.
Hence, by Lemma \ref{lem:blow-upInd}
\[R^i\rho'_{W*} F_{(Y', \tL_1'+\cdots+\tL_r' + E' + E'')\otimes \sW}=0\qfor i\ge 1.\]
Since Theorem \ref{thm:CItsp-loginv} has been proved for $i=0$, we have
\[ \rho'_* F_{(Y', \tL_1'+\cdots+\tL_r' + E' + E'')\otimes \sW} = 
F_{(Y, \tL_1+\cdots+\tL_r + E)\otimes \sW}.\]
Hence we obtain
\eq{thm:blow-upAn1}{R^i\rho_{W*} F_{(Y, \tL_1+\cdots+\tL_r + E)\otimes \sW}= 
R^i(\rho\rho')_{W*} F_{(Y', \tL_1'+\cdots+\tL_r' + \tE + E')\otimes \sW}. }
Denote by $\sigma:\hY \to \A^n$ the blow-up in $Z$ and
$\hL_i\subset \hY$ be the strict transform of $L_i$ and $\Xi=\sigma^{-1}(Z)$.
By Lemma \ref{lem:blow-upInd} we get
\eq{thm:blow-upAn2}{R^i\sigma_{W*} F_{(\hY, \hL_1+\cdots+\hL_r+ \Xi)\otimes \sW}=0 \qfor i\ge 1.}
Denote by $\sigma':\hY'\to \hY$ the blow-up in $\hZ=\sigma^{-1}(0)\subset \Xi$
and $\hL'_i,\; \Xi'\subset \hY'$ be the strict transforms of $\hL_i,\; \Xi$ respectively
and $\Xi''=\sigma'^{-1}(\hZ)$.
Note that $\tZ\subset \hL_3\cap \cdots\cap\hL_n\cap \Xi$ and $\codim(\tZ, \hY)= n-1$ and $\tZ$ intersects transversally with $\hL_1+ \hL_2$.
 Thus by Lemma \ref{lem:blow-upInd} and the case $i=0$ of Theorem \ref{thm:CItsp-loginv}, we obtain
\eq{thm:blow-upAn3}{R\sigma'_{W*}F_{(\hY',\hL_1'+\cdots+\hL_r'+ \Xi'+\Xi'')\otimes \sW}=  F_{(\hY, \hL_1+\cdots+\hL_r+ \Xi)\otimes \sW}.}
Finally, by \cite[Lem. 2.15]{brs}, there is an isomorphism of $\A^n\times W$-schemes 
\eq{thm:blow-upAn4}{ (\hY',\hL_1',\dots,\hL_r,\Xi',\Xi'')\cong
(Y', \tL_1',\dots,\tL_r',E',E'').}
Altogether we obtain for $i\ge 1$
\begin{align*}
R^i\rho_{W*} F_{(Y, \tL_1+\cdots+\tL_r + E)\otimes \sW} 
&= R^i(\rho\rho')_{W*} F_{(Y', \tL_1'+\cdots+\tL_r' + E' + E'')\otimes \sW}, 
                          & &   \text{by }\eqref{thm:blow-upAn1},\\
                        & = R^i(\sigma\sigma')_{W*} F_{(\hY',\hL_1'+\cdots+\hL_r'+ \Xi'+\Xi'')\otimes \sW},
                          & & \text{by }\eqref{thm:blow-upAn4},\\
                       & = R^i\sigma_{W*} F_{(\hY, \hL_1+\cdots+\hL_r+ \Xi)\otimes \sW},
                          & & \text{by }\eqref{thm:blow-upAn3},\\
                                                              & = 0,
                          & & \text{by }\eqref{thm:blow-upAn2}.
\end{align*}
This completes the proof of the proposition.
\end{proof}

\begin{remark}\label{rem:logcohRSC}
For simplicity, we write 
\[\Hlog i(-,F)=\Hlog i(-,\ulomegaCI F)\qfor F\in \RSCNis.\]
By \cite[Cor. 6.8]{rs}, if $\ch(k)=0$ and $F=\Omega^i$, we have 
\[ \Hlog i(-,\Omega^i)=H^i(X,\Omega^i(\log |D|) \qfor (X,D)\in \ulMCorls.\]
Hence $\Hlog i(-,F)$ for $F\in \RSCNis$ is a generalization of cohomology of sheaves of logarithmic differentials.
\end{remark}

\section{Relation with logarithmic sheaves with transfers}\label{log-mottives}

In this section we use the same notations as \cite{bpo}.

Let $\lSm$ be the category of log smooth and separated $\fs$ log schemes of finite type over the base field $k$ and $\SmlSm\subset \lSm$ be the full subcategory consisting of objects whose underlying schemes are smooth over $k$.
Let $\lCor$ be the category with the same objects as $\lSm$ and whose morphisms are log correspondences defined in \cite[Def. 2.1.1]{bpo}.
Let $\lCorSm\subset \lCor$ be the full subcategory consisting of all objects in $\SmlSm$. 

Let $\lPST$ be the category of additive presheaves of abelian groups on $\lCor$ and
$\dNST\subset \lPST$ be the full subcategory consisting of those 
$\sF$ whose restrictions to $\lSm$ are dividing Nisnevich sheaves (see \cite[Def. 3.1.4]{bpo}). It is shown in \cite[\S 4 and Pr. 4.6.6]{bpo} that $\dNST$ is a Grothendieck abelian category and there is an equivalence of categories
\eq{dNSTSm}{ \dNST \simeq \dNST(\SmlSm),}
where the right hand side denotes the full subcategory of the category 
$\lPST(\SmlSm)$ of additive presheaves of abelian groups on $\lCorSm$ consisting of those $\sF$ whose restrictions to $\SmlSm$ are dividing Nisnevich sheaves.
 \medbreak
Now we construct a functor
\eq{LogShv}{ \Log : \uMNSTlog \to \dNST.}
For $\fX=(X,\sM)\in\SmlSm$, we put
$\fXMP=(X, \pX)$, where $\pX\subset X$ is the closed subscheme 
consisting of the points where the log-structure $\sM$ is not trivial. 
By \cite[Theorem III.1.11.12]{ogus}, $\pX$ with reduced structure is a normal crossing divisor on $X$ so that we can view $\fXMP$ as an objects of $\ulMCorls$.
For $F\in \uMPSTlog$ and $\fX\in \SmlSm$, we put
\eq{Flog}{ \Flog(\fX)= F(\fXMP).}
Take $\fY \in \SmlSm$ and 
$\alpha\in \lCor(\fY,\fX)$. By \cite[Def. 2.1.1 and Rem. 2.1.1(iv)]{bpo}, we have
\[\alpha\in \uMCorfin((Y,n\cdot \pY),(X,\pX))\;\text{ for some } n>0,\]
where $n\cdot \pY\hookrightarrow Y$ is the $n$-th thickening of $\pY\hookrightarrow Y$.
By the assumption $F\in \uMPSTlog$, the induced map
\[ \Flog(\fX)=F(\fXMP) \rmapo{\alpha^*}  F(Y,n\cdot \pY) \]
factors through $\Flog(\fY)=F(Y,\pY)\subset F(Y,n\cdot\pY)$ and we get a map
\[ \alog : \Flog(\fX)\to \Flog(\fY).\]
Moreover, for a map $\gamma: F\to G$ in $\uMPSTlog$, the diagram
\[\xymatrix{
\Flog(\fX) \ar[r]^{\gamma}\ar[d]^{\alog} & \Glog(\fX) \ar[d]^{\alog}\\
\Flog(\fY) \ar[r]^{\gamma} & \Glog(\fY) \\}\]
is obviously commutative. 
Hence the assignment $\sX\to \Flog(\sX)$ gives an object $\Flog$ of $\lPST(\SmlSm)$ and we get a functor
\eq{LogPSh}{  \Log : \uMPSTlog \to \lPST(\SmlSm)\;;\; F \to \Flog.}
By the definitions of sheaves (\cite[Def. 1]{kmsyI} and \cite[Def. 3.1.4]{bpo}) and 
\cite[Pr. 1.9.2]{kmsyI}, this induces a functor
\[  \uMNSTlog \to \dNST(\SmlSm)\]
which induces the desired functor \eqref{LogShv} using \eqref{dNSTSm}. 
By the construction, for $F\in \uMNSTlog$ and $\fX\in \SmlSm$ with 
$\sX=\fXMP\in \uMCorls$, we have 
\eq{comparisonlogcoh}{H^i_\Nis(X,F_{\sX}) = H^i_\sNis(\fX,\Flog)\; (\Flog=\Log(F)),}
where the right hand side is the cohomology for the strict Nisnevich topology 
(see \cite[Def. 4.3.1]{bpo}). 

\begin{thm}\label{CItspNis-logDM}
For $F\in \CItspNis$, $\Flog=\Log(F)\in \dNST$ is strictly $\bcube$-invariant in the sense 
\cite[Def. 5.2.2]{bpo}. For $\fX\in \SmlSm$ with 
$\sX=\fXMP\in \uMCorls$, we have a natural isomorphism
\eq{CItspNis-logDM}{ H^i_\Nis(X,F_{\sX})\simeq \Hom_{\lDMeff}(M(\fX),\Flog[i]),}
where $\lDMeff$ is the triangulated category of logarithmic motives defined in \cite[Def. 5.2.3]{bpo}.
\end{thm}
\begin{proof}
We have isomorphisms
\[ H^i_\Nis(X,F_{\sX}) \overset{\eqref{comparisonlogcoh}}{\simeq} 
H^i_\sNis(\fX,\Flog) \overset{(*1)}{\simeq} \colim_{\fY\in \fX_{div}^{Sm}} H^i_\sNis(\fX,\Flog) \overset{(*2)}{\simeq} H^i_\dNis(\fX,\Flog) ,\]
where $(*2)$ comes from \cite[Th. 5.1.8]{bpo} and 
$(*1)$ is a consequence of Theorem \ref{thm:CItsp-loginv} in view of \eqref{comparisonlogcoh}. Hence the strict $\bcube$-invariance of $\Flog$ follows from \cite[Th. 0.6]{shuji}. Finally \eqref{CItspNis-logDM} follows from \cite[Pr. 5.2.8]{bpo}.
\end{proof}

Now we consider the composite functor
\[ \Log': \RSC_\Nis \rmapo{\ulomegaCI} \CItspNis \rmapo{\Log} \dNST.\]

\begin{lemma}\label{lem1;Log}
$\Log$ and $\Log'$ have the same essential image.
\end{lemma}
\begin{proof}
This follows directly from the construction and 
Corollary \ref{cor2:purityreducedmodulus}(3).
\end{proof}

In what follows, we let 
\eq{LogRSC}{ \Log : \RSC_\Nis \to \dNST\;:\; F\to \Flog }
denote $\Log'$ defined as above. By \eqref{Flog}, we have
\eq{FlogRSC}{ \Flog(X,\triv)= F(X)\qfor F\in \RSCNis,\; X\in \Sm,}
where $(X,\triv)$ denotes the log-scheme with the trivial log structure.
\medbreak

\begin{thm}\label{Log-ffexact}
$\Log$ is exact and fully faithful.
\end{thm}
\begin{proof}
First we prove the full faithfulness.
The faithfulness follows from \eqref{FlogRSC}.
Let $F,G\in \RSCNis$ and $\gamma: \Flog\to \Glog$ be a map in $\dNST$.
By \eqref{FlogRSC} it induces maps
$\gamma_X: F(X) \to G(X)$ for all $X\in \Sm$.
They are compatible with the action of $\Cor$ since by \cite[Rem 2.1.3(3)]{bpo},
\[ \Cor(Y,X)= \lCor(Y,\triv),(X,\triv))\qfor X,Y\in \Sm.\]
Thus $\gamma_X$ for $X\in \Sm$ give a map 
$\gamma_{\RSCNis}:F\to G$ in $\RSCNis$.
To see $\Log(\gamma_{\RSCNis})=\gamma$, it suffices by \eqref{dNSTSm} to show that $\Log(\gamma_{\RSCNis})$ and $\gamma$ induce the same map
$\Flog(\fX) \to \Glog(\fX)$ for $\fX\in \SmlSm$.
If $\fX$ has the trivial log-structure, this follows immediately from the construction of $\gamma_{\RSC}$. The general case follows from this in view of the commutative diagram
\[\xymatrix{
\Flog(\fX) \ar[r]^{\gamma} \ar[d]^{j^*} &\Glog(\fX) \ar[d]^{j^*} \\
\Flog(X\backslash\pX,\triv) \ar[r]^{\gamma} &\Glog(X\backslash\pX,\triv) \\}\]
where $j^*$ are induced by the natural map $(X\backslash\pX,\triv)\to \fX$ of log-schemes and injective by the construction and the semipurity of $\ulomegaCI F$. This completes the proof of the full faithfulness.

Next we show the exactness of $\Log$. It suffices to show the following.

\begin{claim}\label{claim;Log-ffexact}
Given an exact sequence $0\to F\to G\to H\to 0$ in $\RSCNis$,
the induced sequence 
\[0\to \Flog(\fX)\to \Glog(\fX)\to H^{\log}(\fX)\to 0\]
is exact for every $\fX\in \SmlSm$ with $X$ henselian local.
\end{claim}

Indeed, by the definition of $\Log$, this is reduced to the exactness of 
\[ 0\to \ulomegaCI F(\fXMP)\to \ulomegaCI G(\fXMP)\to \ulomegaCI H(\fXMP)\to 0,\]
which follows from Corollary \ref{cor2:purityreducedmodulus}(2).
This completes the proof of Theorem \ref{Log-ffexact}.
\end{proof}


\end{document}